\documentclass[11pt]{article}
\usepackage{amsmath,amsthm,amsfonts,amssymb,amscd, amsxtra, mathrsfs,textcomp,verbatim,inputenc}
\usepackage{url}
\usepackage{xcolor}
\usepackage{mathtools}
\usepackage{todonotes} 
\usepackage{cancel}
\usepackage{booktabs}
\usepackage[shortlabels]{enumitem}

\usepackage{soul,colortbl,hhline,rotating,array,multirow,makecell}

\usepackage{graphicx}   % for resizebox
\usepackage{multirow}

\usepackage{subcaption}

\usepackage{soul} % use this (many fancier options)

\usepackage[margin=2.5 cm,nohead]{geometry}
\setlist[enumerate]{label=(\roman*), align=left}
\newtheorem{theorem}{Theorem}
\newtheorem{lemma}[theorem]{Lemma}
\newtheorem{definition}{Definition}
\newtheorem{corollary}[theorem]{Corollary}
\newtheorem{proposition}[theorem]{Proposition}

\newtheorem{remark}{Remark}

\usepackage{graphicx,float}
\usepackage{graphicx}
\usepackage{grffile} % permite nomes de arquivos com espaços

\usepackage{algorithm}
\usepackage{algpseudocode}

\usepackage{tabularx}

\DeclareMathOperator{\grad}{grad}  
  
\usepackage{amssymb}

\algnewcommand{\Input}[1]{%
  \State \textbf{Input:} {\raggedright #1}
}

\algnewcommand{\Initialize}[1]{%
  \State \textbf{Initialize:}
  \Statex \hspace*{\algorithmicindent}\parbox[t]{.8\linewidth}{\raggedright #1}
}

\algnewcommand{\Output}[1]{%
  \State \textbf{Output:} {\raggedright #1}
}

\begin{document}
%%%%%%%%%%%%%%%%%%%%%%%%%%%%%%%%%%%%%%%%%%%%%%%%%%%%%%%%%%%%%
\title{Ball-proximal  point method  on a Hadamard  Manifolds}

\author{
F. Babu   \footnotemark[2] 
\and
 O. P. Ferreira\thanks{Institute of Mathematics and Statistics, Federal University of Goias, Avenida Esperan\c{c}a, s/n, Campus II,  Goi\^ania, GO - 74690-900, Brazil (E-mail: {\tt orizon@ufg.br}, {\tt lfprudente@ufg.br}).}
\and
L. F. Prudente  \footnotemark[1]
\and 
Jen-Chih Yao \thanks{Center for General Education, China Medical University,  Taichung 40402, Taiwan (Email: {\tt yaojc@mail.cmu.edu.tw}, {firoz77b@gmail.com}). }
\and
Xiaopeng Zhao  \thanks{School of Mathematical Sciences, Tiangong University, Tianjin 300387, China. (E-mail: zhaoxiaopeng.2007@163.com)}
}

\maketitle
\begin{abstract}

We consider the problem of minimizing a proper, lower semicontinuous, geodesically convex
function on a Hadamard manifold. Building on ball--proximal (``broximal'') ideas in the
Euclidean setting, viewed as an abstract proximal--type algorithm, we propose and analyse
a Riemannian ball--proximal point method (RB--PPM) whose basic step consists of minimizing
the objective function over a  metric ball centred at the current iterate. We first introduce
the Riemannian broximal map, prove existence and uniqueness of broximal points on Hadamard
manifolds, and derive a KKT--type characterization involving a scalar parameter and the
Riemannian subdifferential. We then show that RB--PPM enjoys a strict decrease of the
squared distance to the solution set whenever the current ball does not contain a minimizer.
This leads to quasi--Fejér monotonicity, finite termination for constant radii, and a
product--form linear decay of the objective values up to the hitting time of the solution set.
We also obtain nonasymptotic complexity bounds for the norms of suitable subgradients and
for the function values, including a linear rate in the number of iterations under constant
radii. Finally, we establish an asymptotic dichotomy,  if the sum of the radii diverges, then
the objective values converge to the optimal value, and, when the solution set is nonempty,
the entire sequence of iterates converges to a minimizer. The resulting scheme provides a geometry aware, ball based analogue of classical Riemannian proximal point methods.

\end{abstract}

\noindent
{\bf Keywords:}  ball-proximal point  method, iteration of complexity,  asymptotic convergence, Hadamard  manifolds.

\medskip
\noindent
{\bf AMS subject classification:}  49J52, 49M15, 65H10, 90C30.

%%%%%%%%%%%%%%%%%%%%%%%%%%%%%%%%%%%
\section{Introduction} 
Throughout the paper we consider the Riemannian optimization problem of the form
\begin{equation} \label{eq:problem}
\min_{p \in \mathcal{M}} f(p),
\end{equation}
where \( \mathcal{M} \) is a Hadamard manifold and 
\( f \colon \mathcal{M} \to \mathbb{R}\cup\{+\infty\} \) 
is a proper, lower semicontinuous, convex function. Recall that a
Hadamard manifold is a complete, simply connected Riemannian manifold with
nonpositive sectional curvature. This class of manifolds offers a particularly
convenient setting for convex analysis beyond Euclidean spaces, namely,  the exponential
map at every point is a global diffeomorphism, any two points are joined by a
unique geodesic, and the squared distance function is convex. 
These geometric features allow one to transplant many constructions from classical
convex optimization to the manifold framework and to design algorithms that take
advantage of the underlying curvature. Problems of the form \eqref{eq:problem}
naturally appear in applications where the variables are modeled in non-Euclidean
domains, such as the cone of symmetric positive definite matrices, flat spaces, and hyperbolic-type models; see, for instance,
\cite{AbsilMahonySepulchre2008,Boumal2020} and the references therein.

Since the seminal works of Martinet and Rockafellar 
\cite{Martinet1970,Rockafellar1976}, the proximal point algorithm (PPA) has become 
a central tool in convex optimization and monotone operator theory, due to its 
robust convergence properties and conceptual simplicity. The basic idea is to 
replace the original problem with  a sequence of regularized subproblems and to use 
their solutions as new iterates. Over the years, numerous variants have been 
developed, combining proximal regularization with projection, splitting, and 
hyperplane techniques; see, for example, the hybrid projection--proximal method of 
Solodov and Svaiter~\cite{SolodovSvaiter2000} and the references therein. In the Riemannian setting, proximal ideas have been successfully adapted to 
Hadamard manifold . The first Riemannian proximal point method for convex 
minimization problems of the form \eqref{eq:problem} was proposed in 
\cite{FerreiraOliveira2002}, and later extended in \cite{LiLopesMartin-Marquez2009} 
to the more general task of finding zeros of monotone vector fields on a Hadamard 
manifolds; see also \cite{Ferreira2006, FerreiraOliveira1998} for related 
developments. In this framework, the classical Riemannian proximal point method 
(R--PPM) computes \(p_{k+1}\) as a solution of a regularized 
subproblem of the form
\begin{equation}\label{eq:ppm}
p_{k+1}\in\arg\min_{q\in\mathcal M}\Big\{ f(q) + \tfrac{\lambda_k}{2} d^2(q,p_k)\Big\},
\end{equation}
where \(\lambda_k>0\) is a penalty parameter. The quadratic term regularizes \(f\)  and at the same time controls the stepsize  through \(\lambda_k\).

Recently, an alternative \emph{ball--proximal} (or \emph{broximal}) viewpoint has
been proposed in Euclidean spaces \cite{gruntkowska2025,gruntkowska2025n}. In the
same conceptual spirit as the classical proximal point method, instead of
regularizing \(f\) with a global quadratic term, the broximal step updates the
iterate as a minimizer of \(f\) over a metric ball around the current point. The
ball constraint plays the role of a local regularization device, keeping the
iterate within a controlled neighbourhood while still allowing for relatively
large moves when the radius is fixed or grows slowly. In the Euclidean setting,
this leads to an abstract proximal--type scheme in which the geometry enters
only through the ball constraint, and the analysis delivers quasi--Fejér
properties, finite termination, and sharp complexity bounds under mild convexity
assumptions, even though the subproblems remain idealized and are mainly used as
a conceptual framework. The idea of minimizing a function over a ball to control the step size and
enforce consistent progress towards a solution also appears in several other
contexts. In Euclidean convex optimization, a ball--minimization primitive has
been used as a powerful oracle for designing accelerated and parallel
algorithms, max--loss minimization schemes, and variants with improved
logarithmic factors and extensions to non--Euclidean geometries; see, e.g.,
\cite{Asi2021,carmon2022,carmon2020,carmon2023,carmon2021,Jambulapati2024,Weigand2024}.
In that line of work, the ball oracle is typically employed as a mechanism for
implementing Monteiro--Svaiter oracles \cite{MonteiroSvaiter2013} under
differentiability and smoothness assumptions. Classical trust--region
algorithms, originating with \cite{Levenberg1944,Marquardt1963} and
systematically developed in \cite{ConnGouldToint2000}, also restrict each step
to a ball around the current iterate, but they do so for a local model of the
objective whose reliability is controlled by the trust--region radius. In
contrast, the BPM viewpoint assumes an exact model and uses the ball purely as a
structural device, thereby explaining why ball constraints arise naturally in
trust--region methods.

 The main purpose of this paper is to introduce and analyse a Riemannian 
ball--proximal point method (RB--PPM) on a Hadamard manifold . The scheme follows 
the broximal approach of \cite{gruntkowska2025,gruntkowska2025n} and is formulated 
entirely in terms of the Riemannian distance, so that all steps and estimates remain 
consistent with the underlying manifold geometry.  Conceptually, our RB--PPM has the same structural features as the Euclidean broximal schemes, namely,  it is an abstract proximal--type method whose iteration map is given by exact minimization over metric balls, providing  a geometry--aware counterpart of these ball--based approaches. Given a current iterate  \(p_k\in\mathcal M\) and a radius \(t_k>0\), the Riemannian broximal step computes
\[
p_{k+1}\in\arg\min\{\,f(q):\,q\in\mathbb{B}(p_k,t_k)\,\},
\]
where \(\mathbb{B}(p_k,t_k)\) denotes the closed metric ball of radius \(t_k\) 
centred at \(p_k\).  On a Hadamard manifold, the Hopf--Rinow theorem ensures that
\(\mathbb{B}(p_k,t_k)\) is compact, so the existence of broximal points follows from
the lower semicontinuity of \(f\). Moreover, if \(f\) is geodesically convex, one
has the following dichotomy:  if \(\mathbb{B}(p_k,t_k)\) intersects the solution
set \(\Omega^*\), then every point in \(\Omega^*\cap\mathbb{B}(p_k,t_k)\) is
broximal;  whereas if \(\mathbb{B}(p_k,t_k)\cap\Omega^*=\varnothing\), there
exists a unique broximal point, which necessarily lies on the boundary
\(\partial\mathbb{B}(p_k,t_k)\). Our main  contributions can be summarized as follows.

\begin{itemize}
\item We define the Riemannian broximal map and the associated RB--PPM, and prove
existence and basic structural properties of broximal points on a Hadamard manifold. In particular, when the
current  metric ball does not intersect the solution set, the broximal point lies
on the boundary sphere and satisfies a KKT-type condition involving a scalar
multiplier and the Riemannian subdifferential of the objective function.

\item Under geodesic convexity, we show that RB--PPM produces a strict reduction in
the squared distance to the solution set whenever the current ball does not contain
a minimizer. This distance decrease implies quasi--Fejér monotonicity with respect
to the solution set, finite termination when the radii are kept constant, and a
linear (product-form) decay of the objective values up to the hitting time of the
solution set.

\item We derive nonasymptotic complexity bounds for both the norms of suitable
subgradients and the objective values, including an inverse–linear rate in the
number of iterations when the radii are constant. These estimates make explicit how
the performance of RB--PPM depends on the chosen sequence of radii and clarify the
trade-off between step length and progress per iteration.

\item We establish an asymptotic dichotomy, namely,  if the sum of the radii along the
iterations diverges, then the objective values converge to the optimal value; and
when the solution set is nonempty, the entire sequence of iterates converges to a
minimizer.
\end{itemize}
The RB--PPM can be viewed as a Riemannian counterpart of the Euclidean broximal 
schemes of \cite{gruntkowska2025,gruntkowska2025n}, and as a geometry--aware 
complement to the classical Riemannian proximal point methods of 
\cite{FerreiraOliveira2002,LiLopesMartin-Marquez2009}. Throughout the paper we 
discuss how the broximal subproblem fits standard Riemannian modelling and how the choice of the radii \(\{t_k\}_{k\in \mathbb N}\) enters the 
nonasymptotic bounds for the distance to the solution set, the function values, and 
suitable subgradient norms.

The remainder of the paper is organized as follows. Section~\ref{sec2} introduces the notation and recalls basic results from Riemannian geometry and subdifferential calculus. Section~\ref{Sec:RB-PPMap} studies the Riemannian broximal map and derives KKT-type optimality conditions for the associated ball-constrained subproblem. Section~\ref{Sec:RB-PPM} presents the RB--PPM and develops its well-posedness, convergence, and complexity analysis. Section~\ref{Sec:NumExp} reports the numerical experiments. Finally, Section~\ref{Sec:SolSupProRB-PPM} contains concluding remarks and directions for future research.

%%%%%%%%%%%%%%%%%%%%%%%%%%%%%%
\section{Notation and terminology} \label{sec2}

In this section, we recall some concepts, notations and basics results about Hadamard manifolds. For more details,  see,   for example, \cite{doCarmo92} and \cite{Sakai96}. Throughout  this paper  $\mathcal M$ represents a finite dimensional Hadamard manifold, 
 $T_p\mathcal M$ the tangent space of $\mathcal M$ at $p$,   $TM=\cup_{p\in M}T_pM$ istangent bundle of $M$ and  ${\cal X}(M)$ the space of smooth vector fields on $M$. The corresponding norm associated to the Riemannian metric $\langle \cdot , \cdot \rangle$
is {represented} by $\lVert\cdot\rVert$. We use $\ell(\gamma)$ to {express} the length of a piecewise smooth curve
$\gamma\colon [a,b] \rightarrow \mathcal M$.  The Riemannian distance between $p$ and $q$ in $\mathcal M$ is  denoted by $d(p,q)$, which induces the original topology on $\mathcal M$, namely, $(\mathcal M, d)$  is a complete metric space and will be denoted by $\mathcal M$. The {\it exponential mapping} $\exp_{p}:T_{p}\mathcal M \rightarrow  \mathcal M $ is defined  by $\exp_{p}v\,=\, \gamma _{p,v}(1)$, where $\gamma _{p,v}$ is the geodesic defined by its  initial position $p$, with velocity $v$ at $p$. Hence, we have $\gamma _{p,v}(t)=\exp_{p}(tv)$. Thus, {\it we will also use the expression $\exp_{p}(tv)$ for denoting  the geodesic   $\gamma_{{p}, v}$ starting  at $p\in \mathcal M$ with velocity $v\in T_p\mathcal M$ at $p$}.  For a $p\in\mathcal M$, the exponential map $\exp_p$ is a diffeomorphism and $\log_p\colon\mathcal M\to T_p\mathcal M$ {indicates} its inverse.  In this case, $d(p,q) = \|\log_pq\|$ holds,  $d(\cdot, q)\colon\mathcal M/\{q\}\to\mathbb{R}$  is $C^{\infty}$  for all $q\in \mathcal M$ and  its gradient is given by ${\rm grad}_{1} d(p, q) = (-\log_pq)/d(q, p)$, for all $q\neq p$, where ${\rm grad}_{1}$ denotes the gradient with respect to first coordinate. In addition,  $d^2(\cdot, q)\colon\mathcal M\to\mathbb{R}$  is $C^{\infty}$  for all $q\in \mathcal M$, and  ${\rm grad}_{1} d^2(p, q) = -2\log_pq$.   Let $\bar{p},\bar{q}\in \mathcal M$ and $(p_{k})_{k\in \mathbb{N}}, (q_{k})_{k\in \mathbb{N}}\subset \mathcal M$ be sequences  such that $\lim_{k\to +\infty} p_{k}=\bar{p}$ and $\lim_{k\to +\infty}q_k=\bar{q}$. Then, for any $q\in M$,  $\lim_{k\to +\infty} \log_{p_{k}}q=  \log_{\bar{p}}q$ and $\lim_{k\to +\infty} \log_qp_{k}= \log_q\bar{p}$ and   $\lim_{k\to +\infty}  \log_{p_{k}}q_k=\log_{\bar{p}}\bar{q}$.  Given $p,q\in\mathcal M$,  {the symbol $\gamma_{pq}(t) := \exp_p\bigl(t\, \log_p q\bigr)$ indicates} the geodesic segment  joining  $p$ to $q$, i.e., $\gamma_{pq}\colon[0,1]\rightarrow\mathcal M$ with $\gamma_{pq}(0)=p$ and $\gamma_{pq}(1)=q$.   Since $M$ is a Hadamard manifold, there exists a unique  geodesic segment  $\gamma_{pq}$  joining $p$ to $q$, its  length  is  equal $d(p,q)$, and the parallel transport along $\gamma_{pq}$ from $p$ to $q$ is denoted by $P_{pq}:T_{p}M\to T_{q}M$.  In the following, we will recall the well-known ``comparison theorem" for triangles in a Hadamard manifold, as stated in \cite[Proposition 4.5]{Sakai96}. 
In the following, we will recall the well-known ``comparison theorem" for triangles in a Hadamard manifold , as stated in \cite[Proposition 4.5]{Sakai96}. 
 \begin{lemma}  \label{le:CosLawF}
Let $\mathcal M$ be a  Hadamard   manifold. The following inequality  holds:
 \begin{equation} \label{eq:coslawF}
d^2({y}, {x})+d^2({y},{z})-2\left\langle  \log_{{y}}{x}, \log_{{y}}{z}\right\rangle\leq d^2({x},{z}),  \qquad   \forall {x}, {y}, {z} \in \mathcal M. 
\end{equation}
In addition, if  the sectional curvature  is $K\equiv 0$ in the whole $\mathcal M$, then \eqref{eq:coslawF} holds as equality.
\end{lemma}
 For any nonempty subset \(C\subset\mathcal M\), we define the {\it distance from  point \(p\in\mathcal M\) to a set \(C\)} by \(\operatorname{dist}(p,C):=\inf_{q\in C} d(p,q).\)  A subset \(C \subset \mathcal M\) is said to be \emph{convex} if for every pair of points \(p, q \in C\), the geodesic segment 
\(
\gamma_{pq}(t) := \exp_p\bigl(t\, \log_p q\bigr),
\)
for all \(t \in [0,1]\), is entirely contained in \(C\). For every closed convex set \(C \subset \mathcal M\) and any point \(q \in \mathcal M\), there exists a \emph{unique} nearest point in \(C\), denoted by
\[
{\cal P}_C(q) := \arg\min_{x \in C} d(q, x),
\]
called the \emph{metric projection} of \(q\) onto \(C\). The following result characterizes this projection; its proof can be found in~\cite{FerreiraOliveira2002}.

\begin{proposition}\label{prop:proj-ineq}
Let \(C \subset \mathcal M\) be closed and convex, and let \(q \in \mathcal M\). Then, for all \(p \in C\), the projection \({\cal P}_C(q)\) is unique and satisfies
\[
\big\langle \log_{{\cal P}_C(q)} q,\; \log_{{\cal P}_C(q)} p \big\rangle \leq 0.
\]
\end{proposition}
The following definition plays an important role in this paper, see \cite[p. 363]{Bourbaki1995}.
\begin{definition} \label{def:linf}
    A function $f\colon\mathcal M\to \mathbb{R}\cup\{+\infty\}$ is said to be
    \emph{lower semi-continuous} (\emph{lsc}), at $p\in \mathcal M$ if  $ \liminf_{q\to p} f(q)= f(p)$.  If $f$ is lower semi-continuous at all points along $\mathcal M$, we simply state that $f$ is \emph{lower semi-continuous}.
\end{definition}
The \emph{domain} of $f\colon\mathcal M \to \mathbb{R}\cup\{+\infty\}$ is represented by
${\rm dom} (f) \coloneqq \{ p\in \mathcal M\ : \ f(p) < +\infty\}$.
The function $f$ is said to be \emph{convex} (resp.\ \emph{strictly convex}) if, for any $p,q\in \mathcal M$, the function
$f\circ{\gamma_{pq}}\colon[0, 1]\to\mathbb{R}$ is convex (resp.\ strictly convex), i.e.,
$(f\circ{\gamma_{pq}})(t)\leq(1-t)f(p)+tf(q)$ (resp.\ $(f\circ{\gamma_{pq}})(t)<(1-t)f(p)+tf(q)$),
for all $t\in[0,1]$.
A function $f\colon\mathcal M \to \mathbb{R}\cup\{+\infty\}$ is said to be
\emph{$\sigma$-strongly convex} for $\sigma > 0$ if, for any
$p,q\in {\rm dom}(f)$, the composition $f\circ{\gamma_{pq}}\colon[0, 1]\to \mathbb{R}\cup\{+\infty\}$ is
$\sigma$-strongly convex, i.e.,
$(f\circ{\gamma_{pq}})(t)\leq(1-t)f(p)+tf(q)-\frac{\sigma}{2}t(1-t)d^2(q,p)$, for all $t\in[0,1]$.
It is worth noting that, by \cite[Lemma~3 and Theorem~2]{Ferreira2008}, whenever $f$ is convex,
the set ${\rm dom}(f)$ is convex.

\begin{definition}
    The \emph{subdifferential} of a proper, convex function
    $f\colon\mathcal M \to \mathbb{R}\cup\{+\infty\}$ at $p\in \mathcal {\rm dom} (f) $ is
 $
        \partial f(p)
        \coloneqq
        \bigl\{
            {s} \in T_p\mathcal M:~
            f(q) \geq f(p)+\langle {s},\exp^{-1}_pq \rangle,~\text{for all } q\in \mathcal M
        \bigr\}.
    $
\end{definition}
The proof of the first item of the following theorem can be found in \cite[Theorem 4.10, p. 76]{Udriste1994}, while the proof of the second one follows the same idea as the first one.
\begin{proposition} \label{pr:f-convex-subdiff}
    Let $f\colon\mathcal M\to \mathbb{R}$ be a function. Then,
    \begin{enumerate}
        \item
        \label{thm:f-convex-subdiffi}
        The function $f$ is convex if and only if
        $f(p)\geq f(q) + \langle {s}, \log_qp\rangle$,
        for all $p, q\in \mathcal M$ and all $s\in \partial f(q)$.
        \item
        \label{thm:f-convex-subdiffii}
        The function $f$ is $\sigma$-strongly convex if and only if
        $f(p)\geq f(q) + \langle {s}, \log_qp \rangle + \frac{\sigma}{2}d^2(p,q)$,
        for all $p, q\in \mathcal M$ and all ${s}\in \partial f(q)$.
    \end{enumerate}
\end{proposition}

The proof of the following result is an immediate consequence of   \cite[Proposition 2.5]{WangLiWangYao2015}.
\begin{proposition}
    \label{cont_subdif}
    Let $f\colon \mathcal M \rightarrow \mathbb{R}\cup\{+\infty\}$ be a convex
    and lower semi-continuous function. Consider the sequence
    $\{p_{k}\}_{k\in\mathbb{N}}\subset {\rm intdom} (f)$ such that $\displaystyle\lim_{k\to\infty}p_{k}={\bar p} \in{\rm intdom}  (f)$.
    If $\{s_{k}\}_{k\in\mathbb{N}}$ is a sequence such that $s_{k}\in \partial f(p_{k})$
    for every $k\in \mathbb{N}$, then $\{s_{k}\}_{k\in\mathbb{N}}$ is bounded and
    its cluster points belong to $\partial f({\bar p}).$
\end{proposition}

Next  we  present  the concept of  Fej\'er  convergence, which plays  an important role in the asymptotic convergence analysis  of the  methods studied in this paper.
\begin{definition}\label{def:QuasiFejer}
Let $\{y_{k}\}_{k\in\mathbb{N}}$ be a sequence in the complete metric space $\mathcal{M}$.
We say that $\{y_{k}\}_{k\in\mathbb{N}}$ is \emph{Fejér convergent} to a set $W\subset \mathcal{M}$
if \(d^{2}(y_{k+1},w)\leq d^{2}(y_{k},w)\), for all  \(w\in W\)  and all  \(k\in \mathbb{N}\).
\end{definition}

The main property of  Fej\'er convergent   sequences is stated in the next result, and its proof is similar to  that in \cite[Proposition~3.2 and Theorem 3.3]{CombettesVu2013} with the Euclidean distance replaced by the Riemannian one.
\begin{theorem}\label{teo.qf}
	Let $\{y_{k}\}_{k\in\mathbb{N}}$ be a sequence in a complete metric space $\mathcal{M}$.  If $\{y_{k}\}_{k\in\mathbb{N}}$ is Fej\'er  convergente to a nonempty set $W\subset  \mathcal{M}$, then $\{y_{k}\}_{k\in\mathbb{N}}$ is bounded. Furthermore, if a  cluster point of $\{y_{k}\}_{k\in\mathbb{N}}$ belongs to $W$, then there exists ${\bar y}\in W$ such that $\lim_{k\rightarrow\infty}y_k={\bar y}$.
\end{theorem}

%%%%%%%%%%%%%%%%%%%%%%%%%%%%
\subsection{Relative interior of a convex set}

In this section we adapt the classical notion of relative interior to
geodesically convex subsets of a Hadamard manifold \(\mathcal{M}\), in a way
suited to our broximal analysis. We first recall the notion of a totally
geodesic submanifold, which plays the role of an ``affine hull'' in this
setting, and then define the relative interior of a convex set
\(C\subset\mathcal{M}\) as its interior in the topology induced by
\(\operatorname{aff}(C)\).

We say that a submanifold \(S\subset \mathcal{M}\) is \emph{totally geodesic}
if every geodesic of \(S\), with the induced metric, is also a geodesic of \(\mathcal{M}\).

\begin{definition}\label{def:ri}
Let $C\subset\mathcal{M}$ be a nonempty geodesically convex set. 
The \emph{geodesic affine hull} of $C$ is the smallest totally geodesic 
submanifold of $\mathcal{M}$, denoted by  ${\operatorname{aff}(C)}$ that contains $C$, i.e.,
\[
{\operatorname{aff}(C)}=\bigcap\{S\subset\mathcal{M}:\ S \text{ is a totally geodesic  submanifold and } C\subset S\,\}.
\]
The \emph{relative interior} of $C$, denoted by $\operatorname{ri}(C)$, is the 
interior of $C$ with respect to the topology of ${\operatorname{aff}(C)}$:
\[
\operatorname{ri}(C)=\{\,x\in C : \exists\,r>0 \text{ such that } \mathbb{B}(x,r)\cap {\operatorname{aff}(C)} \subset C\,\}.
\]
Equivalently, $x\in\operatorname{ri}(C)$ if and only if there exists $r>0$ such that
\(
\mathbb{B}_{\operatorname{aff}(C)}(x,r)\subset C,
\)
where $\mathbb{B}_{\operatorname{aff}(C)}(x,r):=\{y\in {\operatorname{aff}(C)}:\ d(x,y)<r\}$ denotes the open metric ball of radius $r$
centred at $x$, taken in the induced Riemannian metric on $\mathcal{N}$.
\end{definition}

We recall that, for a given set $C\subset{\cal M}$, the \emph{closure} of $C$ in $\mathcal{M}$ is denoted by ${\rm cl \,}{C}$ .
\begin{lemma}\label{lem:ri-density}
Let $C\subset\mathcal{M}$ be a nonempty geodesically convex set and assume that
$\operatorname{ri}(C)\neq\varnothing$. Then $\operatorname{ri}(C)$ is dense in
$C$, i.e.,
\(
{\rm cl\,}\operatorname{ri}(C)\;=\;{\rm cl\,}C.
\)
\end{lemma}
\begin{proof}
Since $\operatorname{ri}(C)\neq\varnothing$, by definition of the relative
interior there exist ${\bar x}\in C$ and $r_0>0$ such that
\begin{equation}\label{eq:ri-ball-x0}
\mathbb{B}_{\operatorname{aff}(C)}({\bar x},r_0)\ \subset\ C,
\end{equation}
where $\mathbb{B}_{\operatorname{aff}(C)}({\bar x},r_0):=\{z\in\operatorname{aff}(C):~ d({\bar x},z)<r_0\}$ is the open metric ball in
$\operatorname{aff}(C)$. In particular, ${\bar x}\in\operatorname{ri}(C)$. Let $y\in C$ with $y\neq{\bar x}$ and let
\(\gamma_{{\bar x}y}:[0,1]\to\operatorname{aff}(C)\) be the unique geodesic
segment joining ${\bar x}$ to $y$. By the geodesic convexity of $C$ we have
\(\gamma_{{\bar x}y}([0,1])\subset C\). We first prove that
\begin{equation}\label{eq:gamma-ri}
\gamma_{{\bar x}y}(t)\ \in\ \operatorname{ri}(C)\qquad\forall\,t\in[0,1).
\end{equation}
Fix $t\in(0,1)$ and set $x_t:=\gamma_{{\bar x}y}(t)$. Since
$\operatorname{aff}(C)$ is a totally geodesic submanifold of the Hadamard
manifold $\mathcal{M}$, it is itself a Hadamard manifold, with the induced
metric and distance. For each $u\in\operatorname{aff}(C)$, let
$\sigma_{uy}:[0,1]\to\operatorname{aff}(C)$ denote the unique geodesic segment
joining $u$ to $y$. Define, for fixed $t\in(0,1)$, the map
\[
\Phi_t:\operatorname{aff}(C)\to\operatorname{aff}(C),\qquad
\Phi_t(u):=\sigma_{uy}(t).
\]
By the smooth dependence of geodesics on their endpoints, $\Phi_t$ is smooth on
$\operatorname{aff}(C)$ and satisfies $\Phi_t({\bar x})=\gamma_{{\bar x}y}(t)
=x_t$. Because $\operatorname{aff}(C)$ is Hadamard, there are no conjugate points
along $\gamma_{{\bar x}y}$. Standard Jacobi–field arguments then imply that the
differential
\(
D\Phi_t({\bar x}):T_{\bar x}\operatorname{aff}(C)\to T_{x_t}\operatorname{aff}(C)
\)
is an isomorphism for every $t\in(0,1)$. Hence, by the inverse function
theorem, there exist open neighbourhoods $U_t$ of ${\bar x}$ and $W_t$ of
$x_t$ in $\operatorname{aff}(C)$ such that
\(
\Phi_t:U_t \to W_t
\)
is a diffeomorphism. Shrinking $U_t$ if necessary, we may assume that
\(
U_t \subset \mathbb{B}_{\operatorname{aff}(C)}({\bar x},r_0)  \subset C,
\)
where the last inclusion follows from \eqref{eq:ri-ball-x0}. For any $u\in U_t$, both $u$ and $y$ belong to $C$, and $C$ is geodesically
convex; hence the geodesic segment $\sigma_{uy}([0,1])$ is contained in $C$.
In particular, we have \(\Phi_t(u)=\sigma_{uy}(t)\in C\), for all \(u\in U_t\). Thus, 
\(
W_t=\Phi_t(U_t)\ \subset\ C.
\)
Since $W_t$ is an open neighbourhood of $x_t$ in $\operatorname{aff}(C)$ and
$W_t\subset C$, the definition of the relative interior implies that
$x_t\in\operatorname{ri}(C)$. As $t\in(0,1)$ is arbitrary, we obtain
\eqref{eq:gamma-ri}. We now proceed to  prove the density statement. The inclusion
\(
{\rm cl\,}\operatorname{ri}(C)\subset{\rm cl\,}C
\)
is immediate from $\operatorname{ri}(C)\subset C$. For the reverse inclusion,
let $y\in C$ be arbitrary. If $y={\bar x}$, then $y\in\operatorname{ri}(C)$ and
hence $y\in{\rm cl\,}\operatorname{ri}(C)$. If $y\neq{\bar x}$, consider the
geodesic $\gamma_{{\bar x}y}$ as above and define $t_n:=1-\frac{1}{n}$ for
$n\ge2$. Then $t_n\in[0,1)$, $\lim_{n\to +\infty}t_n=1$ and
\(
\lim_{n\to +\infty}\gamma_{{\bar x}y}(t_n)=\gamma_{{\bar x}y}(1)=y
\)
By \eqref{eq:gamma-ri}, $\gamma_{{\bar x}y}(t_n)\in\operatorname{ri}(C)$ for all $n$, so
$y\in{\rm cl\,}\operatorname{ri}(C)$. Since $y\in C$ is  arbitrary, we obtain
\(
C\subset{\rm cl\,}\operatorname{ri}(C).
\)
Both sides yields
\(
{\rm cl\,}C\subset{\rm cl\,}\operatorname{ri}(C).
\)
Combining the two inclusions
\(
{\rm cl\,}\operatorname{ri}(C)\subset{\rm cl\,}C
\)
and
\(
{\rm cl\,}C\subset{\rm cl\,}\operatorname{ri}(C),
\)
we conclude that
\(
{\rm cl\,}\operatorname{ri}(C)\;=\;{\rm cl\,}C,
\)
as claimed.
\end{proof}

\begin{lemma}\label{lem:ri-ball-intersect}
Let $\mathcal{M}$ be a finite-dimensional Hadamard manifold with distance $d$,
and let $f:\mathcal{M}\to\mathbb{R}\cup\{+\infty\}$ be proper, lower semicontinuous
and geodesically convex, with $\operatorname{ri}(\operatorname{dom}f)\neq\varnothing$.
Fix $p\in\operatorname{dom}f$ and $t>0$. Then
\(
\operatorname{ri}\big(\mathbb{B}(p,t)\big)\ \cap\
\operatorname{ri}\big(\operatorname{dom}f\big)\ \neq\ \varnothing.
\)
\end{lemma}
\begin{proof}
By \cite[Lemma~3 and Theorem~2]{Ferreira2008}, the set $\operatorname{dom}f$ is
geodesically convex. Since $\operatorname{ri}(\operatorname{dom}f)\neq\varnothing$,
we can apply Lemma~\ref{lem:ri-density} with $C=\operatorname{dom}f$ to  obtain
\(
{\rm cl\,}\operatorname{ri}(\operatorname{dom}f)={\rm cl\,}\operatorname{dom}f.
\)
In particular,
\[
\operatorname{dom}f\ \subset\ {\rm cl\,}\operatorname{dom}f=
{\rm cl\,}\operatorname{ri}(\operatorname{dom}f).
\]
Thus,  $p\in{\rm cl\,}\operatorname{ri}(\operatorname{dom}f)$. Hence there exists a
sequence $\{z_k\}_{k\in\mathbb{N}}\subset\operatorname{ri}(\operatorname{dom}f)$
such that $\lim_{k\to +\infty}z_k= p$. On the other hand, by the definition of the relative interior and the fact that
$\mathbb{B}(p,t)$ has full geodesic affine hull, we have
\(
\operatorname{ri}\big(\mathbb{B}(p,t)\big)=\{\,q\in\mathcal{M}: d(p,q)<t\,\},
\)
so $p\in{\rm cl\,}\operatorname{ri}(\mathbb{B}(p,t))$. Since
$\operatorname{ri}(\mathbb{B}(p,t))$ is open in $\mathcal{M}$ and
$\lim_{k\to +\infty}z_k= p\in{\rm cl\,}\operatorname{ri}(\mathbb{B}(p,t))$, there exists
${\bar k}\in\mathbb{N}$ such that
\(
z_k\in\operatorname{ri}\big(\mathbb{B}(p,t)\big), 
\)
for all $k\geq {\bar k}$. For such $k$ we have simultaneously $z_k\in\operatorname{ri}\big(\operatorname{dom}f\big)$ and $z_k\in\operatorname{ri}\big(\mathbb{B}(p,t)\big)$, 
which shows that
\(
\operatorname{ri}\big(\mathbb{B}(p,t)\big)\cap
\operatorname{ri}\big(\operatorname{dom}f\big)\neq\varnothing.
\)
This completes the proof.
\end{proof}

%%%%%%%%%%%%%%%%%%%%%%%%%%%%%%%
\section{Riemannian ball-proximal  map} \label{Sec:RB-PPMap}
This section introduces the {\it ball–proximal (broximal) map} in the Riemannian setting, obtained by minimizing the objective over a metric ball. This construction is the manifold analogue of Euclidean ball–proximal steps  and will form the basis for the Riemannian ball-proximal   point method.  Section~\ref{ses:BroxSup} analyzes the ball-proximal   subproblem,  we derive the corresponding optimality conditions, establishs the existence of ball-proximal   points, and describe how their location depends on whether the ball intersects the solution set. In Section~\ref{ses:PropBroxMap} we develop some structural properties of the ball-proximal   map, including boundary location and uniqueness when the ball does not intersect the solution set. These results provide the foundation needed to introduce the algorithm and carry out its convergence analysis in the next section.

Let \( \mathcal{M} \) be a complete Riemannian manifold with associated Riemannian distance \(d\).
For \(p\in\mathcal{M}\) and \(t>0\), denote the {\it closed metric ball } and its boundary, {\it the intrinsic sphere},  by
\[
\mathbb{B}(p,t):=\{\,q\in\mathcal{M}:\ d(p,q)\le t\,\}, 
\qquad 
\mathbb{S}(p,t):=\{\,q\in\mathcal{M}:\ d(p,q)=t\,\}.
\]
For a extended-real-valued proper, lower semicontinuous, and geodesically convex  function  \(f:\mathcal{M}\to\mathbb{R}\cup\{+\infty\}\), define the set of {\it minimizers and the optimal value} by
\begin{equation}\label{eq:solset}
{\Omega^*}:=\arg\min_{p\in\mathcal{M}} f(p),\qquad \qquad  {f^*}:=\inf_{p\in\mathcal{M}} f(p).
\end{equation}

{\it Throughout this paper we do not assume that $\Omega^*$ is nonempty, except when explicitly stated. However, we do assume that the optimal value $f^*>-\infty$.} {\it  In addition, we assume that $\operatorname{ri}(\operatorname{dom}f)\neq\varnothing$.} Next, we introduce the \emph{ball-proximal} map, the Riemannian counterpart of the Euclidean operator introduced in \cite{gruntkowska2025, gruntkowska2025n}. 
\begin{definition}\label{def:Rbrox}
Let $f:\mathcal{M}\to\mathbb{R}\cup\{+\infty\}$ be a function  and $t>0$. The  \emph{ball-proximal (broximal)  map} of $f$ with 
radius $t$ is defined, for each $p\in\mathcal{M}$, by
\begin{equation} \label{eq:DefRbrox}
\operatorname{brox}^{t}_{f}(p)
:=\arg\min\{\,f(q):\,q\in\mathbb{B}(p,t)\,\}.
\end{equation} 
\end{definition}

In general, $\operatorname{brox}^{t}_{f}(\cdot)$ is set-valued. On a complete Riemannian manifold,
Hopf–Rinow implies that every closed metric ball $\mathbb{B}(p,t)$ is compact; thus, if $f$ is proper and
lower semicontinuous, it attains its minimum on $\mathbb{B}(p,t)$, so $\operatorname{brox}^{t}_{f}(p)\neq\varnothing$
for all $p$ and $t>0$. On a Hadamard manifold, if $f$ is geodesically convex, then the geodesic convexity
of $\mathbb{B}(p,t)$ further implies the uniqueness of the ball-proximal   point under natural conditions, for example
when the ball does not intersect the solution set $\Omega^*$. These facts are made precise in the next two
theorems.

%%%%%%%%%%%%%%%%%%%%%%%%%%%%%%%
\subsection{The  ball-proximal  subproblem} \label{ses:BroxSup}

In this section we recall a few basic notions and results from convex analysis on manifolds (see, for example, \cite{LiLopesMartin-Marquez2009,LiMordukhovichWang2011}) and then state KKT-type optimality conditions for the  {\it ball-proximal   subproblem} 
\begin{equation} \label{eq:Bsubprof}
\min\{\,f(q):~q\in\mathbb{B}(p,t) \}, 
\end{equation} 
 that defines the ball-proximal   map in ~\eqref{eq:DefRbrox}. Let $C\subset\mathcal{M}$ be a closed, geodesically convex set. 
Its {\it normal cone} at $q\in C$ is defined by 
\begin{equation*}
N_C(q)\ :=\ \big\{\eta\in T_q\mathcal{M}\;:\;\langle \eta,\log_q y\rangle \le 0\ \ \forall\,y\in C\big\},
\end{equation*}
where $\log_q y$ denotes the Riemannian logarithm (the initial velocity of the minimizing geodesic from $q$ to $y$).  We set $N_C(q)=\varnothing$ for $q\notin C$. The \emph{indicator function} of $C$ is the mapping $\delta_C:\mathcal{M}\to\mathbb{R}\cup\{+\infty\}$ defined by
\[
\delta_C(q):=
\begin{cases}
0, & \text{if } q\in C,\\[0.3ex]
+\infty, & \text{if } q\notin C.
\end{cases}
\]
The indicator satisfies $\operatorname{dom}\delta_C=C$, so $\delta_C$ is proper if and only if $C\neq\varnothing$. If $C$ is closed, then $\delta_C$ is lower semicontinuous, and if $C$ is geodesically convex, then $\delta_C$ is geodesically convex. Moreover, for $q\in C$ its subdifferential equals the normal cone, i.e., 
\(
\partial\delta_C(q)=N_C(q),
\)
while $\partial\delta_C(q)=\varnothing$ for $q\notin C$.  Using the notions introduced above, we now  state  a result that is a consequence of the results  of \cite{LiLopesMartin-Marquez2009}.
\begin{lemma}\label{le:visp}
Let $\mathcal{M}$ be a Hadamard manifold, let $f:\mathcal{M}\to\mathbb{R}\cup\{+\infty\}$ be proper, lower semicontinuous, and geodesically convex, and let $C\subset\mathcal{M}$ be a closed, geodesically convex set. 
Assume the standard constraint qualification $\operatorname{ri}(\operatorname{dom}f)\cap\operatorname{ri}(C)\neq\varnothing$. 
Then, 
\[
\partial\big(f+\delta_C\big)(p)=\partial f(p)\;+\;N_C(p), \qquad \forall p\in C\cap \operatorname{ri}(\operatorname{dom}f).
\]
\end{lemma}
We now state the KKT-type optimality conditions for the minimizing a geodesically convex function over a closed metric ball. 
The result below characterizes the solution through the associated multiplier and its relation to the subdifferential of~$f$. This result plays a central role in the analysis of the ball-proximal   subproblem.
\begin{theorem}\label{thm:KKT-ball}
Let $\mathcal{M}$ be a Hadamard manifold with distance $d$, let $f:\mathcal{M}\to\mathbb{R}\cup\{+\infty\}$ be proper, lower semicontinuous, and geodesically convex.  Take $p\in (\operatorname{dom}f)$ and $t>0$, and define the closed metric ball 
\(
\mathbb{B}(p,t):=\{\,q\in\mathcal{M}:\ d(p,q)\le t\,\}.
\)
Then ${p_{+}} \in\mathbb{B}(p,t)$ solves the ball-proximal   subproblem  \eqref{eq:Bsubprof} if and only if there exists a multiplier $ \theta_t(p) \ge 0$ such that the following KKT-conditions hold:
\begin{equation} \label{eq:KKT-ball-stat}
 \mathbf{0}\in \partial f({p_{+}} ) \;+\;  \theta_t(p) \big(-\log_{{p_{+}} }p\big), \qquad d(p,{p_{+}} )\le t, \qquad   \theta_t(p) \big(d(p,{p_{+}} )-t\big)=0, 
\end{equation}
Equivalently, the stationarity condition in  \eqref{eq:KKT-ball-stat} is that there exists a multiplier \(\theta_t(p) \ge 0\) such that 
\(
 \theta_t(p) \log_{{p_{+}} }p \in \partial f({p_{+}} ).
\)
\end{theorem}
\begin{proof}
Let $C:=\mathbb{B}(p,t):=\{\,x\in\mathcal{M}:\ d(p,x)-t \leq 0 \}$. First we note that, due to  $p\in \operatorname{dom}f$, applying Lemma~\ref{lem:ri-ball-intersect} we conclude  that    \(\operatorname{ri}(\operatorname{dom}f)\cap\operatorname{int}\mathbb{B}(p,t)\neq\varnothing\). Then, since $\delta_C$ is proper, lower semicontinuous, and geodesically convex with 
$\partial\delta_C(q)=N_C(q)$ for $q\in C$, by Lemma~\ref{le:visp} we have 
\begin{equation}\label{eq:opt-char}
{p_{+}} \in C\ \text{solves}\ \min_{q\in C} f(q)\quad\Longleftrightarrow\quad
{0}\in \partial\big(f+\delta_C\big)({p_{+}} )=\partial f({p_{+}} )+N_C({p_{+}} ).
\end{equation}
Therefore the KKT system reduces to computing the normal cone $N_C(q)$ of the closed metric ball. If $d(p,q)<t$, then $q$ is an interior point of $C$ and $N_C(q)=\{0\}$. Now, assume that  $d(p,q)=t$. On a Hadamard manifold, the distance function $x\mapsto d(p,x)$ is convex and smooth away from $p$, with
\[
{\rm gard}\ d(p,\cdot)(q)=(-{\log_q p})/{d(p,q)}, \qquad q\neq p.
\]
Since $d(p,q)=t>0$, we have ${\rm gard} g(q)={\rm gard} d(p,\cdot)(q)=-\log_q p/t$.  Therefore,  the normal cone is 
\[
N_C(q)=\big\{\ \mu\,{\rm gard} g(q):\ \mu\geq 0 \big\}=\big\{\ \mu\big((-{\log_q p})/{t}\big):\ \mu\geq 0\ \big\}
=\big\{\theta (-\log_q p):\ \theta \geq 0 \big\}.
\]
Summarizing, we conclude that 
\begin{equation}\label{eq:NCball}
N_{\mathbb{B}(p,t)}(q)=
\begin{cases}
\{\mathbf{0}\}, & d(p,q)<t,\\[0.4ex]
\{\theta(-\log_q p):\ \theta\ge 0\,\}, & d(p,q)=t,
\end{cases}
\qquad\text{and }N_C(q)=\varnothing\ \text{ if }q\notin C.
\end{equation}
Combining \eqref{eq:opt-char} with \eqref{eq:NCball} and the feasibility $d(p,{p_{+}} )\le t$ yields:
there exists $ \theta_t(p) \geq 0$ such that
\[
\mathbf{0}\in \partial f({p_{+}} )+ \theta_t(p) (-\log_{{p_{+}} }p),
\]
with $ \theta_t(p) =0$ when $d(p,{p_{+}} )<t$ and $ \theta_t(p) \ge0$ satisfying stationarity when $d(p,{p_{+}} )=t$. This is precisely the stationarity condition \eqref{eq:KKT-ball-stat}.
The complementary slackness condition in  \eqref{eq:KKT-ball-stat} encodes the dichotomy between the interior case ($d(p,{p_{+}} )<t\Rightarrow  \theta_t(p) =0$) and the boundary case ($d(p,{p_{+}} )=t$ with $ \theta_t(p) \ge0$).  Primal feasibility in \eqref{eq:KKT-ball-stat} is part of the statement ${p_{+}} \in C$.
Finally, the last statement follows immediately from the stationarity condition \eqref{eq:KKT-ball-stat}, which is equivalently written as the existence of a multiplier $ \theta_t(p) \ge 0$ such that
\(
 \theta_t(p) \log_{{p_{+}} }p \in \partial f({p_{+}} ).
\)
This completes the proof.
\end{proof}

\begin{remark}\label{re:KKT-ball-cases}
The KKT conditions in Theorem~\ref{thm:KKT-ball} cleanly split the optimality structure into two mutually exclusive cases:
\begin{itemize}
\item[(i)]   If $d(p,{p_{+}} )<t$, the ball constraint is inactive; hence complementary slackness forces $ \theta_t(p) =0$, and the stationarity condition reduces to $\mathbf{0}\in\partial f({p_{+}} )$.
\item[(ii)] If $d(p,{p_{+}} )=t$, the constraint is active and there exists a multiplier $ \theta_t(p) \ge 0$ such that the stationarity condition reads $\mathbf{0}\in\partial f({p_{+}} )+ \theta_t(p) (-\log_{{p_{+}} }p)$, i.e., there is $ \theta_t(p) \log_{{p_{+}} }p \in\partial f({p_{+}} )$.
\end{itemize}
In particular, interior solutions are unconstrained minimizers of $f$, whereas boundary solutions balance a subgradient of $f$ against the outward normal to the metric ball at ${p_{+}} $ (represented by $\log_{{p_{+}} }p$), as dictated by complementary slackness.
\end{remark}

%%%%%%%%%%%%%%%%%%%%%%%%%%%%%%%%%%%%
\subsection{Properties of  ball-proximal  map} \label{ses:PropBroxMap} 

This Section develops the fundamental structural properties of the ball-proximal   map.  
Using the KKT characterization derived in Section~\ref{ses:BroxSup}, we establish
existence, boundary location, and when the metric ball does not intersect the solution
set,  uniqueness of ball-proximal   points. These results clarify the geometric behaviour of
ball-proximal   steps and will serve as key ingredients for the Riemannian ball-proximal   point
method introduced later.

\begin{theorem}\label{thm:first-brox-Riem}
Let \(\mathcal{M}\) be a Hadamard manifold and let \(f:\mathcal{M}\to\mathbb{R}\cup\{+\infty\}\) be proper, lower semicontinuous, and geodesically convex. Fix \(p\in\operatorname{dom}f\) and \(t>0\). Then, 
\begin{enumerate}[label=\textup{(\roman*)}]
\item The ball-proximal   set
\(
\operatorname{brox}^{t}_{f}(p)
\)
is nonempty. Moreover, if \(\mathbb{B}(p,t)\cap {\Omega^*}\neq\varnothing\), then \(\operatorname{brox}^{t}_{f}(p)\subseteq {\Omega^*}\).
\item  If \(\mathbb{B}(p,t)\cap {\Omega^*}=\varnothing\), then \(\operatorname{brox}^{t}_{f}(p)=\{{p_{+}}\}\) is a singleton and  \({p_{+}}\in \mathbb{S}(p,t)\), i.e., \(d(p,{p_{+}})=t\).
\end{enumerate}
\end{theorem}

\begin{proof}
To prove (i), note that, by the Hopf–Rinow theorem, the closed metric ball \(\mathbb{B}(p,t)\) is compact. Since \(f\) is proper and lower semicontinuous, \(f\) attains its minimum on \(\mathbb{B}(p,t)\); hence \(\operatorname{brox}^{t}_{f}(p)\neq\varnothing\). If \(\mathbb{B}(p,t)\cap {\Omega^*}\neq\varnothing\), then the minimum value of \(f\) over \(\mathbb{B}(p,t)\) equals \(f_*\) as defined in~\eqref{eq:solset}, so every ball-proximal   point belongs to \({\Omega^*}\).

For proving  (ii), let \({p_{+}}\in \operatorname{brox}^{t}_{f}(p)\). We first show that $y$ must lie on the boundary
sphere $\mathbb{S}(p,t)$. Suppose, by contradiction, that $y\in \operatorname{int}\mathbb{B}(p,t)$. Then there exists $\varepsilon>0$ such that
$\mathbb{B}({p_{+}},\varepsilon)\subset\mathbb{B}(p,t)$. Since $y$ minimizes $f$ on $\mathbb{B}(p,t)$, it also minimizes $f$ on $\mathbb{B}({p_{+}},\varepsilon)$.  Hence $y$ is a local minimizer of $f$ on $\mathcal{M}$. Because $f$ is geodesically convex, any local minimizer is a global minimizer. Indeed, if there
exists $x\in\mathcal{M}$ with $f(x)<f({p_{+}})$, let $\gamma_{{p_{+}}x}:[0,1]\to\mathcal{M}$ be the
minimizing geodesic from $y$ to $x$. By geodesic convexity, we have 
\[
f(\gamma_{{p_{+}}x}(\lambda)) \leq  (1-\lambda)f({p_{+}})+\lambda f(x) < f({p_{+}}), 
\qquad\text{for all sufficiently small }\lambda>0,
\]
which contradicts the local minimality of $y$.
Thus $y$ is a global minimizer of $f$, which implies that  $y\in\Omega^*$ and, in particular,
${p_{+}}\in\mathbb{B}(p,t)\cap\Omega^*$, which contradicts the hypothesis
$\mathbb{B}(p,t)\cap\Omega^*=\varnothing$. Therefore every ball-proximal   point must satisfy
${p_{+}}\in\mathbb{S}(p,t)$ and the first statement of item (ii) is proved.

For uniqueness, suppose by contradiction that there exist ${p_{+}}\neq {{\hat p}^+}$ with
${p_{+}},{{\hat p}^+}\in\operatorname{brox}^{t}_{f}(p)\subset\mathbb{S}(p,t)$. Let
$\gamma:[0,1]\to\mathcal{M}$ be the minimizing geodesic from ${p_{+}}$ to ${{\hat p}^+}$. Since $\mathcal{M}$
is Hadamard, the ball $\mathbb{B}(p,t)$ is geodesically convex, so
$\gamma([0,1])\subset\mathbb{B}(p,t)$ and the points $\gamma(\lambda)$ with $\lambda\in(0,1)$
belong to $\operatorname{int}\mathbb{B}(p,t)$. Denote
\(
m := f({p_{+}})=f({{\hat p}^+}) = \min_{q\in\mathbb{B}(p,t)} f(q).
\)
By geodesic convexity of $f$, for any $\lambda\in(0,1)$,
\[
f(\gamma(\lambda))
\;\le\; (1-\lambda)f({p_{+}})+\lambda f({{\hat p}^+})
= m.
\]
Since $m$ is the minimum value of $f$ on $\mathbb{B}(p,t)$, we must have
$f(\gamma(\lambda))=m$ for all $\lambda\in(0,1)$, so every interior point of the segment
$\gamma([0,1])$ is also a minimizer of $f$ on $\mathbb{B}(p,t)$. In particular, pick
$\lambda\in(0,1)$ and set $z:=\gamma(\lambda)$. Then $z\in\operatorname{brox}^{t}_{f}(p)$ and
$z\in\operatorname{int}\mathbb{B}(p,t)$, which contradicts the boundary property established
above. Hence such ${p_{+}}\neq {{\hat p}^+}$ cannot exist, and $\operatorname{brox}^{t}_{f}(p)$ is a
singleton with its unique element on $\mathbb{S}(p,t)$.
\end{proof}

\begin{theorem}\label{thm:second-brox-Riem}
Let \(\mathcal{M}\) be a Hadamard manifold and let \(f:\mathcal{M}\to\mathbb{R}\cup\{+\infty\}\) be proper, lower semicontinuous, and geodesically convex. Fix \(p\in\operatorname{dom} f\), \(t>0\), and let \({p_{+}}\in \operatorname{brox}^{t}_{f}(p)\). Then there exists a scalar \(\theta_t(p)\ge 0\) such that
\begin{enumerate}[label=\textup{(\roman*)}]
\item \(\theta_t(p)\,\log_{p_{+}}p \in \partial f({p_{+}})\);
\item \(f(x) \ge f({p_{+}})+\theta_t(p) \big\langle \log_{p_{+}}p,\ \log_{p_{+}}x\big\rangle \), for all \(x\in\mathcal{M}\).
\end{enumerate}
Moreover, if \(\mathbb{B}(p,t)\cap {\Omega^*}\neq\varnothing\), then \(\theta_t(p)=0\); if \(\mathbb{B}(p,t)\cap {\Omega^*}=\varnothing\), then \(d(p,{p_{+}})=t\) and \(\theta_t(p)>0\).
\end{theorem}
\begin{proof}
We proceed by cases: either \(\mathbb{B}(p,t)\cap {\Omega^*}\neq\varnothing\) or \(\mathbb{B}(p,t)\cap {\Omega^*}=\varnothing\). First we assume that \(\mathbb{B}(p,t)\cap {\Omega^*}\neq\varnothing\). By Theorem~\ref{thm:first-brox-Riem}\,(i), we have \({p_{+}}\in {\Omega^*}\). Hence, by Proposition~\ref{pr:f-convex-subdiff}\,(i),
\({0}\in\partial f({p_{+}})\). Taking \(\theta_t(p)=0\) proves \textup{(i)}, and \textup{(ii)} follows from \(f(x)\ge f({p_{+}})\) for all \(x\in\mathcal{M}\).

Now, we assume that  \(\mathbb{B}(p,t)\cap {\Omega^*}=\varnothing\).
By Theorem~\ref{thm:first-brox-Riem}\,(ii), \({p_{+}}\in\mathbb{S}(p,t)\) and \({p_{+}}\) minimizes \(f\) over the geodesically convex set \(\mathbb{B}(p,t)\).  Therefore, it follows from Theorem~\ref{thm:KKT-ball} that there exists \(\theta_t(p)\geq 0\) such that 
\(
\theta_t(p) \log_{p_{+}}p  \in  \partial f({p_{+}}).
\)
 The subgradient inequality in Proposition~\ref{pr:f-convex-subdiff}  then gives, for all \(x\in\mathcal{M}\), 
\[
f(x)\geq  f({p_{+}})+\langle s,\log_{p_{+}}x\rangle=f({p_{+}})+\theta_t(p)\,\big\langle \log_{p_{+}}p,\ \log_{p_{+}}x\big\rangle,
\]
which is \textup{(ii)}. Finally, in the disjoint case \({p_{+}}\notin {\Omega^*}\), so \(\theta_t(p) \log_{p_{+}}p  \neq\mathbf{0}\), whence  \(\theta_t(p)>0\). Moreover \(d(p,{p_{+}})=t\) by Theorem~\ref{thm:first-brox-Riem}\,(ii).
\end{proof}

\begin{corollary}\label{co:ctc}
Suppose the assumptions of Theorem~\ref{thm:second-brox-Riem} hold.
Let $p\in\operatorname{dom} f$, fix $t>0$, and take
${p_{+}} \in\operatorname{brox}^{t}_{f}(p)$.
Then, 
\begin{enumerate}[label=\textup{(\roman*)}]
\item If $p^*\in\Omega^*$ with ${p_{+}}\neq p^*$ and  $p\notin\Omega^*$, then
\begin{equation}\label{eq:ct-lb-riem}
\theta_t(p)\geq \frac{f({p_{+}})-f^*}{\,d(p,{p_{+}})\,d({p_{+}},p^*)\,}.
\end{equation}
\item If $\mathbb{B}(p,t)\cap\Omega^*=\varnothing$ and $z\in\operatorname{brox}^{t}_{f}({p_{+}})$, then
\begin{equation}\label{eq:theta-chain-riem}
\theta_t({p_{+}})\,d({p_{+}},z)\leq \theta_t(p)\,d(p,{p_{+}}).
\end{equation}
\end{enumerate}
\end{corollary}
\begin{proof}
(i)  Apply Theorem~\ref{thm:second-brox-Riem}\,(ii) with $x=p^*$, we have 
\(
f(p^*)-f({p_{+}})\geq \theta_t(p)\,\big\langle \log_{p_{+}} p,\ \log_{p_{+}} p^*\big\rangle.
\)
Since $f(p^*)=f^*$ and ${p_{+}}\neq p^*$, rearranging and using Cauchy--Schwarz in $T_{p_{+}}\mathcal M$ yields
\[
f({p_{+}})-f^*\leq  \theta_t(p)\,\|\log_{p_{+}} p\|\,\|\log_{p_{+}} p^*\|
=\theta_t(p)\,d(p,{p_{+}})\,d({p_{+}},p^*).
\]
If $\mathbb{B}(p,t)\cap\Omega^*\neq\varnothing$, then $f({p_{+}})=f^*$ and \eqref{eq:ct-lb-riem} holds trivially.
Otherwise, by Theorem~\ref{thm:first-brox-Riem}\,(ii) we have $d(p,{p_{+}})=t>0$, so division is legitimate and
\eqref{eq:ct-lb-riem} follows.

(ii) If $\mathbb{B}({p_{+}},t)\cap{\Omega^*}\neq\varnothing$, then due to $z\in\operatorname{brox}^{t}_{f}({p_{+}})$, by Theorem~\ref{thm:first-brox-Riem}\,(i) we have 
$z\in{\Omega^*}$ and, by Theorem~\ref{thm:second-brox-Riem}, $\theta_t({p_{+}})=0$ and then  \eqref{eq:theta-chain-riem} holds trivially. 
Hence assume $\mathbb{B}({p_{+}},t)\cap{\Omega^*}=\varnothing$; then $z\neq {p_{+}}$ and, by Theorem~\ref{thm:first-brox-Riem}\,(ii), $z\in\mathbb{S}({p_{+}},t)$. Hence, applying  Theorem~\ref{thm:second-brox-Riem}\,(ii) to the pair $(p,{p_{+}})=({p_{+}},z)$ with $x={p_{+}}$, we conclude that 
\[
f({p_{+}})-f(z)\geq \theta_t({p_{+}})\,\big\langle \log_z {p_{+}},\ \log_z {p_{+}}\big\rangle
=\theta_t({p_{+}})\,d({p_{+}},z)^2.
\]
On the other hand, applying  Theorem~\ref{thm:second-brox-Riem}\,(ii) to the pair $(p,{p_{+}})=(p,{p_{+}})$ with $x=z$, we obtain that 
\(
f(z)-f({p_{+}})\geq \theta_t(p)\langle \log_{p_{+}} p,\ \log_{p_{+}} z \rangle.
\)
Hence, using Cauchy--Schwarz in $T_{p_{+}}\mathcal M$ we have 
\[
f({p_{+}})-f(z)\leq  -\,\theta_t(p)\,\big\langle \log_{p_{+}} p,\ \log_{p_{+}} z\big\rangle
\leq  \theta_t(p)\,\|\log_{p_{+}} p\|\,\|\log_{p_{+}} z\|
=\theta_t(p)\,d(p,{p_{+}})\,d({p_{+}},z).
\]
Comparing the two above bounds  and dividing by $d({p_{+}},z)>0$ yields  \eqref{eq:theta-chain-riem}, concluding the proof.
\end{proof}

%%%%%%%%%%%%%%%%%%%%%%%%%%%%%%%
\section{Riemannian ball-proximal  point method} \label{Sec:RB-PPM}

In this section, we introduce the Riemannian version of the ball–proximal  point method  for solving problem~\eqref{eq:problem}. We define the ball-proximal  map on metric balls, present the algorithmic scheme, and establish its basic well-posedness and first-order properties, which form the foundation for the convergence and complexity analyses developed in the next sections.

In the following we present the conceptual {\it Riemannian  version of the ball-proximal  point method  (RB-PPM)} for solving   the  unconstained  optimization problem~\eqref{eq:problem}.

\begin{algorithm}[H]
\begin{footnotesize}
\begin{description}
\item[Step 0.]
Take a sequence of radii \( \{t_k\}_{k\in\mathbb{N}}\subset(0,+\infty) \) and an initial point \( p_0\in {\rm dom}f \). Set \( k\gets 0 \).

\item[Step 1.]
Given the current iterate \(p_k\in\mathcal{M}\), compute \(p_{k+1}\in\mathcal{M}\) such that
\begin{equation} \label{eq:bpmsub}
p_{k+1}\in \operatorname{brox}^{t_k}_{f}(p_{k}) :=\arg\min\{\,f(q):~ q\in \mathbb{B}(p_k,t_k)\,\}.
\end{equation}

\item[Step 2.]
If \(p_{k+1}=p_k\), then \textbf{stop} and return \(p_k\); otherwise, set \(k\gets k+1\) and go to \textbf{Step~1}.
\end{description}
\caption{Riemannian ball-proximal  point method (RB-PPM)}
\label{Alg:RB-PPM-exact}
\end{footnotesize}
\end{algorithm}
Each iteration minimizes \(f\) over the metric ball \(\mathbb{B}(p_k,t_k)\), producing \(p_{k+1}\). The {\it sequence of radii}  \( \{t_k\}_{k\in\mathbb{N}}\subset(0,+\infty) \) will be fixed later in the convergence analysis; for now we only require \(t_k>0\). If \(p_{k+1}\in \operatorname{int}\mathbb{B}(p_k,t_k)\), then \( 0\in\partial f(p_{k+1})\); otherwise \(p_{k+1}\in \mathbb{S}(p_k,t_k)\) and satisfies the boundary first-order (KKT) condition
\begin{equation*}
0\ \in\ \partial f({p_{+}})\ +\ \lambda\,\grad \mathrm{d}_p({p_{+}})\qquad\text{for some }\lambda\ge 0,
\end{equation*}
where \(\mathrm{d}_p(q):=d(p,q)\). By Theorem~\ref{thm:first-brox-Riem}\,(ii), when \(\mathbb{B}(p_k,t_k)\cap {\Omega^*}=\varnothing\) the ball-proximal   point is unique and lies on the boundary, so the step length equals the radius,
\(
d(p_{k+1},p_k)=t_k,
\)
and the method keeps advancing until the ball captures a global minimizer. The stopping rule \(p_{k+1}=p_k\) means that \(p_k\) is a fixed point of the ball-proximal   map, i.e., \(p_{k}\in \operatorname{brox}^{t_k}_{f}(p_{k})\). Therefore, from Theorem~\ref{thm:first-brox-Riem},  for any initial point $p_0\in\operatorname{dom}f$ and any choice of sequence of radii  $\{t_k\}_{k\in\mathbb N}\subset(0,+\infty)$, the ball-proximal   subproblem \eqref{eq:bpmsub} at each iteration admits a solution and hence the update $p_{k+1}\in \operatorname{dom} f$ is well defined; we shall denote the resulting {\it RB-PPM sequence} by $\{p_k\}_{k\in\mathbb N}$. In the following we state and prove some properties of the RB-PPM sequence. 

First note  that by Theorem~\ref{thm:second-brox-Riem}, for each  sequence of radii \( \{t_k\}_{k\in\mathbb{N}}\subset(0,+\infty) \) there exists a sequence  \( \{\theta_{t_k}(p_k)\}_{k\in\mathbb{N}}\subset[0,+\infty) \)  such that 
\begin{equation*} 
f(x)\geq f(p_{k+1})+ \theta_t(p_k)\,\big\langle \log_{p_{k+1}}p_k,\ \log_{p_{k+1}}x\big\rangle, \qquad \forall x\in\mathcal{M},    \quad  \forall k\in\mathbb{N}.
\end{equation*} 
Using these dual scalars, we now establish a basic monotonicity property that will be instrumental later, namely,  the “scaled multipliers” \(t_k\,\theta_{t_k}(p_k)\) cannot increase along RB-PPM and hence \(\theta_{t_k}(p_k)\) itself is nonincreasing when the radius is fixed.
\begin{lemma} \label{le:sqthek}
The nonnegative sequence \( \{t_k\theta_{t_k}(p_k)\}_{k\in\mathbb{N}} \) is nonincreasing. As  a consequence, if \(t_k\equiv t>0\) is constant, then \( \{\theta_{t_k}(p_k)\}_{k\in\mathbb{N}} \) is nonincreasing.
\end{lemma}
\begin{proof}
If the algorithm reaches \( \Omega^* \) at some index \(\hat k\), then by Theorem~\ref{thm:second-brox-Riem} we have \(\theta_{t_{\hat k}}(p_{\hat k})=0\), and by the optimality of subsequent iterates the sequence \(\{\theta_{t_k}(p_k)\}_{k\geq \hat k}\) remains identically zero; thus the claim is trivial from that index on. Otherwise, consider two consecutive boundary steps \(k\) and \(k+1\), so that
\(p_{k+1}\in\operatorname{brox}^{t_k}_{f}(p_k)\) and
\(p_{k+2}\in\operatorname{brox}^{t_{k+1}}_{f}(p_{k+1})\).
Applying Corollary~\ref{co:ctc}\,(ii) to the points \(p=p_k, p_{+}=p_{k+1}\) and \(z=p_{k+2}\) with  radius \(t_k=d(p, p_{+})\) and  \(t_{k+1}=d(p_{+}, z)\),   yields
\(
\theta_{t_{k+1}}(p_{k+1}) t_{k+1}\leq \theta_{t_k}(p_k)t_k.
\)
Therefore,  \( \{t_k\theta_{t_k}(p_k)\}_{k\in\mathbb{N}} \) is nonincreasing, as stated and the first  statement  is proved. The proof of the second statement is an immediate  consequence of the first one. 
\end{proof}

Next we present  a version of the classical inequality for the proximal method   adapted to the ball-proximal   point method in the  Hadamard manifold, see \cite[Lemma1.1]{CorreaLemarechal1993} and see also \cite[Lema 6.2]{FerreiraOliveira2002}. 

\begin{lemma} \label{le:coslaw}
Assume that   $\mathbb B(p_k,t_k)\cap {\Omega^*}= \varnothing$. Then,  $p_{k+1}\notin {\Omega^*}$, \(d(p_{k+1},p_k)=t_k\) and \(\theta_{t_k}(p_k)>0\). In addition, there holds
\begin{equation} \label{eq:coslaw}
d^2(p_{k+1},{p}) \leq d^2(p_k,{p}) - d^2(p_k,p_{k+1})+ \frac{2}{\theta_{t_k}(p_k)}(f({p})-f(p_{k+1})),  \qquad \forall p\in\mathcal{M}.
\end{equation} 
\end{lemma} 
\begin{proof}
We first apply  Theorem~\ref{thm:second-brox-Riem}(ii) with \(t=t_k\),  \(p=p_k\), \(p_{+}=p_{k+1}\) and \(x={p} \in \mathcal{M}\), to have
\begin{equation} \label{eq:ficp}
f({p})-f(p_{k+1})\geq \theta_{t_k}(p_k) \langle \log_{p_{k+1}}p_k,\ \log_{p_{k+1}}{p}\rangle.
\end{equation} 
In addition, because  $\mathbb B(p_k,t_k)\cap {\Omega^*}= \varnothing$, then Theorem~\ref{thm:second-brox-Riem}  implies that $p_{k+1}\notin {\Omega^*}$ and \(\theta_{t_k}(p_k)>0\). In addition,  we can apply Lemma~\ref{le:CosLawF} with \(x=p_k\), \(y=p_{k+1}\) and  \(z=p\) to obtain that 
\begin{equation}\label{eq:hinge}
d^2(p_k,{p})\geq d^2(p_k,p_{k+1}) + d^2(p_{k+1},{p}) - 2\big\langle \log_{p_{k+1}}p_k,\ \log_{p_{k+1}}{p}\big\rangle.
\end{equation}
Thus, combining \eqref{eq:ficp} with \eqref{eq:hinge} and taking into account that \(\theta_{t_k}(p_k)>0\) we obtain \eqref{eq:coslaw}.
\end{proof}

\begin{corollary} \label{eq:monfv}
The sequence \(\{f(p_k)\}_{k\in\mathbb N}\) is monotone nonincreasing; that is,
\(f(p_{k+1})\le f(p_k)\) for all \(k\in\mathbb N\).
\end{corollary} 
\begin{proof}
If \(\mathbb B(p_k,t_k)\cap\Omega^*\neq\varnothing\), then by Theorem~\ref{thm:first-brox-Riem}\,(i) we have \(p_{k+1}\in\Omega^*\), hence \(f(p_{k+1})=f^*\le f(p_k)\). If \(\mathbb B(p_k,t_k)\cap\Omega^*=\varnothing\), we can apply Lemma~\ref{le:coslaw} with \(p=p_k\)  to obtain 
\(
f(p_k)-f(p_{k+1})\geq  \theta_{t_k}(p_k)\,t_k^2> 0,
\)
which implies \(f(p_{k+1})<f(p_k)\). Therefore, in both cases \(f(p_{k+1})\le f(p_k)\) for all \(k\), proving that \(\{f(p_k)\}_{k\in\mathbb N}\) is monotone nonincreasing.
\end{proof}

We next recall a basic geometric fact that will be used throughout the paper,  if the ball $\mathbb{B}(p_k,t_k)$ meets $\Omega^*$, then the ball–proximal subproblem returns an optimal point $p_{k+1}\in\Omega^*$; otherwise, the squared distance to $\Omega^*$ decreases by at least $t_k^2$, yielding Fejér convergence of $\{p_k\}_{k\in\mathbb{N}}$ to   $\Omega^*$.

\begin{lemma}\label{le:RB-PPM-convex}
 The following statements hold:
\begin{itemize}
\item[\textup{(i)}]  If $\mathbb B(p_k,t_k)\cap {\Omega^*}\neq \varnothing$, then any ball-proximal   point is optimal, i.e.,  $p_{k+1}\in {\Omega^*}$. In particular, if $t_0\ge d(p_0,{\Omega^*})$, then  RB-PPM finds a minimizer of problem \eqref{eq:problem}   in one iteration.
\item[\textup{(ii)}]  If $\mathbb B(p_k,t_k)\cap {\Omega^*}= \varnothing$ and \({\Omega^*}\neq \varnothing\), then  for every ${p^*}\in {\Omega^*}$, there holds
\begin{equation}\label{eq:distance-drop}
d^2(p_{k+1},{p}_*)\leq d^2(p_k,{p}_*) - t_k^2.
\end{equation}
As a consequence,   
\begin{equation} \label{eq:ctkb}
\operatorname{dist}^2(p_{k+1},{\Omega^*})\leq  \operatorname{dist}^2(p_k,{\Omega^*}) - t_k^2. 
\end{equation} 
\end{itemize}
Consequently,  if ${\Omega^*}\neq \varnothing$ and $\mathbb B(p_k,t_k)\cap {\Omega^*}= \varnothing$ for all \(k\in {\mathbb N}\), then the sequence  $\{p_k\}_{k\in {\mathbb N}}$  is Fej\'er convergent   to ${\Omega^*}$ and the sequence $\{\operatorname{dist}(p_k,{\Omega^*})\}_{k\in {\mathbb N}}$ is nonincreasing.
\end{lemma} 
\begin{proof}
Proof of item \textup{(i)}: The first statement follows by a direct application of Theorem~\ref{thm:first-brox-Riem}\textup{(i)}. For the second statement, observe that $t_0 \ge d(p_0,\Omega^*)$ implies $\mathbb{B}(p_0,t_0)\cap\Omega^*\neq\emptyset$. Thus, by the first statement, the first iterate $p_1$ belongs to $\Omega^*$ and is therefore a minimizer of \eqref{eq:problem}.

Proof of item \textup{(ii)}:  We first note that due   $\mathbb B(p_k,t_k)\cap {\Omega^*}= \varnothing$ we can apply Lemma~\ref{le:coslaw} with $p=p^*$ to obtain that 
\begin{equation} \label{eq:coslawap1}
d^2(p_{k+1},{p^*}) \leq d^2(p_k,{p^*}) - d^2(p_k,p_{k+1})+ \frac{2}{\theta_{t_k}(p_k)}(f({p^*})-f(p_{k+1})).
\end{equation} 
Since \(\theta_{t_k}(p_k)>0\) and   \(f({p}_*)=f_*\le f(p_{k+1})\),  \eqref{eq:coslawap1} implies that  \(d^2(p_{k+1},{p}_*) \leq d^2(p_k,{p}_*) - d^2(p_k,p_{k+1})\). Thus,  using also that  \(d(p_k,p_{k+1})=t_k\) yields the desired inequality \eqref{eq:distance-drop}. In addition,  taking into account that  \eqref{eq:distance-drop}  holds for every \({p}_*\in {\Omega^*}\) and \( \operatorname{dist}^2(p_{k+1},{\Omega^*}) \leq d^2(p_{k+1},{p}_*)\),  taking the best choice on the right hand side of \eqref{eq:distance-drop}  gives \eqref{eq:ctkb},  which  proves the second statement in item (ii).

We proceed to prove the last statement. For that, we assume that $\Omega^*\neq\varnothing$ and fix an arbitrary $p_\ast\in\Omega^*$. Since   $\mathbb B(p_k,t_k)\cap\Omega^*=\varnothing$, for every $p_\ast\in\Omega^*$  we have 
\[
d^2(p_{k+1},p_\ast)\leq  d^2(p_k,p_\ast)-t_k^2\leq  d^2(p_k,p_\ast),
\]
which implies that 
\(
d(p_{k+1},p_\ast)\le d(p_k,p_\ast)\),  for all \(k\in {\mathbb N}\). Therefore,  the sequence  $\{p_k\}_{k\in {\mathbb N}}$  is Fej\'er convergent   to ${\Omega^*}$. Finally, taking the infimum over $p_\ast\in\Omega^*$ gives
\(
\operatorname{dist}(p_{k+1},\Omega^*)\leq  \operatorname{dist}(p_k,\Omega^*), 
\)
for all \(k\in {\mathbb N}\).
Then,  $\{\operatorname{dist}(p_k,\Omega^*)\}_{k\in\mathbb N}$ is nonincreasing., and the proof is concluded. 
\end{proof}

The next lemma provides a quantitative estimate of the decrease in  function values between two consecutive iterates, relating it to the distance from the current iterate to the solution set.
\begin{lemma}\label{eq:ppicb}
For any choice of ${p}_*\in {\Omega^*}$, the following inequality holds
\begin{equation}\label{eq:qlinear1}
f(p_{k+1})-f^* \leq \frac{d(p_{k+1},p_*)}{d(p_{k+1},p_*)+t_k}\bigl(f(p_k)-f^*\bigr),
\qquad \forall k\in \mathbb{N}.
\end{equation}
As a consequence,
\begin{equation}\label{eq:qlinear2}
f(p_{k+1})-f^* \leq
\frac{\operatorname{dist}(p_{k+1},\Omega^*)}{\operatorname{dist}(p_{k+1},\Omega^*)+t_k}
\bigl(f(p_k)-f^*\bigr),
\qquad \forall k\in \mathbb{N}.
\end{equation}
\end{lemma}

\begin{proof}
First observe that \eqref{eq:qlinear1} is immediate when $p_{k+1}\in\Omega^*$, since in this case $f(p_{k+1})=f^*$ and the right-hand side of \eqref{eq:qlinear1} is nonnegative. Now assume that $p_{k+1}\notin\Omega^*$. Then, by
Theorem~\ref{thm:first-brox-Riem}\textup{(i)–(ii)}, we have $\mathbb{B}(p_k,t_k)\cap\Omega^*=\varnothing$, whence $p_{k+1}\in\mathbb{S}(p_k,t_k)$, i.e., $d(p_k,p_{k+1})=t_k$, and $\theta_{t_k}(p_k)>0$. Applying Theorem~\ref{thm:second-brox-Riem} with $t=t_k$, $x=p_k$, $p=p_k$ and $p_+=p_{k+1}$ yields
\[
f(p_k)-f(p_{k+1})
 \geq \theta_{t_k}(p_k)\,\big\langle \log_{p_{k+1}}p_k,\ \log_{p_{k+1}}p_k\big\rangle
=\theta_{t_k}(p_k)\,d^2(p_{k+1},p_k)
=\theta_{t_k}(p_k)\,t_k^2,
\]
which implies that
\begin{equation}\label{eq:valdrop2}
f(p_{k+1})-f^* \leq f(p_k)-f^* - \theta_{t_k}(p_k)\,t_k^2 .
\end{equation}
On the other hand, applying Theorem~\ref{thm:second-brox-Riem} with
$t=t_k$, $x=p_*$, $p=p_k$ and $p_+=p_{k+1}$, where $p_*\in\Omega^*$, we obtain
\(
f(p_*)-f(p_{k+1}) \geq
\theta_{t_k}(p_k)\,\big\langle \log_{p_{k+1}}p_k,\ \log_{p_{k+1}}p_*\big\rangle.
\)
Since $f(p_{k+1})>f^*=f(p_*)$ in the present case, this is equivalent to
\[
0 < f(p_{k+1})-f^* \leq
-\theta_{t_k}(p_k)\,\big\langle \log_{p_{k+1}}p_k,\ \log_{p_{k+1}}p_*\big\rangle.
\]
Because $\theta_{t_k}(p_k)>0$, by the Cauchy–Schwarz inequality in
$T_{p_{k+1}}\mathcal{M}$ we deduce
\[
f(p_{k+1})-f^* \leq
\theta_{t_k}(p_k)\,\|\log_{p_{k+1}}p_k\|\,\|\log_{p_{k+1}}p_*\|
=
\theta_{t_k}(p_k)\,t_k\,d(p_{k+1},p_*).
\]
Since $d(p_{k+1},p_*)>0$, this inequality is equivalent to
\begin{equation}\label{eq:theta-lb}
\theta_{t_k}(p_k)\,t_k^2 \geq
\frac{(f(p_{k+1})-f^*)\,t_k}{d(p_{k+1},p_*)}>0.
\end{equation}
Substituting \eqref{eq:theta-lb} into \eqref{eq:valdrop2} yields
\[
f(p_{k+1})-f^* \leq f(p_k)-f^*
 - \frac{t_k}{d(p_{k+1},p_*)}\bigl(f(p_{k+1})-f^*\bigr),
\]
which is easily seen to be equivalent to the desired inequality
\eqref{eq:qlinear1}. Finally, since \eqref{eq:qlinear1} holds for all
$p_*\in\Omega^*$, taking the infimum over $p_*\in\Omega^*$ on the right-hand side
gives \eqref{eq:qlinear2}. The proof is complete.
\end{proof}

%%%%%%%%%%%%%%%%%%%%%%%%%%%%%%%%%
\subsection{Iteration complexity analysis} \label{sec:IntComplAnalysis}
In this section  we establish iteration–complexity results for the RB–PPM. We first derive a finite–termination criterion expressed in terms of the cumulative sum  of the squared radii. We then obtain quantitative subgradient estimates, in particular  a monotonicity property and a summation inequality, which lead to an explicit $\mathcal{O}(1/K)$ bound for constant radii. Finally, we prove value–decrease inequalities that, under a fixed radius, yield a geometric (linear) rate of convergence. We begin with the finite–termination result.

\begin{theorem}\label{thm:icfv}
Assume that $\Omega^*\neq\varnothing$. 
If $\sum_{j=0}^{K-1} t_j^2 \geq d^2(p_0,{\Omega^*})$, then $p_K\in {\Omega^*}$. In particular, if $t_k\equiv t>0$ is constant, then RB-PPM reaches ${\Omega^*}$ in at most
\[
K \leq \left \lceil \frac{d^2(p_0,{\Omega^*})}{t^2}\right\rceil
\]
iterations, where  $\lceil\cdot\rceil$ maps a real number to the smallest integer greater than or equal to it.
\end{theorem}
\begin{proof}
Assume, by contradiction, that the point  \(p_K\notin\Omega^*\). Then, by Lemma~\ref{le:RB-PPM-convex}(i), we must have
\(\mathbb{B}(p_j,t_j)\cap\Omega^*=\varnothing\) for every \(j=0,1,\ldots,K-1\). Summing the decrease inequality
\eqref{eq:ctkb} in Lemma~\ref{le:RB-PPM-convex}(ii) over these indices yields
\[
\operatorname{dist}^2(p_K,\Omega^*) \leq \operatorname{dist}^2(p_0,\Omega^*)-\sum_{j=0}^{K-1} t_j^2 \leq 0.
\]
But \(p_K\notin\Omega^*\) implies \(\operatorname{dist}^2(p_K,\Omega^*)>0\). Thus considering that \(
\sum_{j=0}^{K-1} t_j^2 \geq \operatorname{dist}^2(p_0,\Omega^*)
\), we have  a contradiction. Hence \(p_K\in\Omega^*\). As a consequence, if \(t_k\equiv t>0\), each boundary step subtracts \(t^2\) from the right-hand side of
\eqref{eq:ctkb}, then  there can be at most
\(\big\lceil \operatorname{dist}^2(p_0,\Omega^*)/t^2\big\rceil\) boundary steps before some ball \(\mathbb{B}(p_k,t)\) intersects \(\Omega^*\), which   proves  the  desired inequality.
\end{proof} 

We next quantify the subgradients produced by RB-PPM, their norms are nonincreasing and their
weighted sum is bounded by the initial gap, which yields an explicit $\mathcal{O}(1/K)$ bound when the radius is constant.

\begin{theorem}\label{thm:subgrad-sum}
Let  \(\theta_{t}\) be  given by Theorem~\ref{thm:second-brox-Riem}.  Define the subgradient sequence as follows 
\begin{equation} \label{eq:dsgic}
s_{k+1}\ :=\ \theta_{t_k}(p_k)\,\log_{p_{k+1}}p_k\ \in\ \partial f(p_{k+1}), \qquad \forall k\in \mathbb{N}, 
\end{equation} 
Then, the following  hold:

\begin{enumerate}[label=\textup{(\roman*)}]
\item The sequence of norms \(\{\|s_{k}\|\}_{k\ge1}\) is nonincreasing. Moreover, if there exists $K\ge0$ such that
$p_{K+1}\in\operatorname{int}\mathbb B(p_{K},t_{K})$, equivalently, $0\in\partial f(p_{K+1})$ and $\theta_{t_{K}}(p_{K})=0$, 
then $s_j=\mathbf 0$ for every $j\ge K+1$.
\item  For every \(K\ge1\), there holds
\begin{equation}\label{eq:nonsmooth-sum}
\sum_{k=0}^{K-1} t_k\,\|s_{k+1}\|\leq  f(p_0)-f^*.
\end{equation}
In particular,  if \(t_k\equiv t>0\), then for all \(K\ge1\), we have 
\begin{equation}\label{eq:nonsmooth-last}
\|s_{K}\|\leq  \frac{f(p_0)-f^*}{t} \frac{1}{K}.
\end{equation}
\end{enumerate}
\end{theorem}
\begin{proof}
To prove item (i),  fix $k\ge0$. We prove that $\|s_{k+2}\|\le \|s_{k+1}\|$. Two cases are possible, \(\mathbb{B}(p_k,t_k)\cap\Omega^*\neq\varnothing\) or \( \mathbb{B}(p_k,t_k)\cap\Omega^*=\varnothing\). If $\mathbb{B}(p_k,t_k)\cap\Omega^*\neq\varnothing$, then by
Theorem~\ref{thm:first-brox-Riem}\,(i) any ball-proximal   point lies in $\Omega^*$. Hence,  we have  $p_{k+1}\in\Omega^*$, 
$0\in\partial f(p_{k+1})$ and $\theta_{t_k}(p_k)=0$. Therefore, 
\(
\|s_{k+1}\|=\theta_{t_k}(p_k)\,d(p_k,p_{k+1})=0.
\)
Thus, applying  Corollary~\ref{co:ctc}\,(ii)  with \(p=p_k\), \(p_{+}=p_{k+1}\) and \(z=p_{k+2}\) yields
\[
\theta_{t_{k+1}}(p_{k+1})\,t_{k+1}\leq  \theta_{t_k}(p_k)\,t_k= 0.
\]
Hence, we conclude that  $\theta_{t_{k+1}}(p_{k+1})=0$ and, in particular, $\|s_{k+2}\|=0=\|s_{k+1}\|$. Thus the claim holds in case \(\mathbb{B}(p_k,t_k)\cap\Omega^*\neq\varnothing\). Now, we  assume that \( \mathbb{B}(p_k,t_k)\cap\Omega^*=\varnothing\). 
In this case, the step is a boundary step and
\[
d(p_{k+1},p_k)=t_k\qquad\text{and}\qquad \|s_{k+1}\|=\theta_{t_k}(p_k)\,t_k.
\]
If the next step $k+1$ is also a boundary step, then $d(p_{k+2},p_{k+1})=t_{k+1}$ and $\|s_{k+2}\|=\theta_{t_{k+1}}(p_{k+1})\,t_{k+1}$. Applying  Corollary~\ref{co:ctc}\,(ii) with  \(p=p_k\), \(p_{+}=p_{k+1}\)  $t_k$ and \(z=p_{k+2}\)  gives
\(
\theta_{t_{k+1}}(p_{k+1})\,t_{k+1}\leq  \theta_{t_k}(p_k)\,t_k,
\)
that is, by definition \eqref{eq:dsgic} we have 
\[
\|s_{k+2}\|\leq  \|s_{k+1}\|.
\]
If instead the step $k+1$ is interior, then $\theta_{t_{k+1}}(p_{k+1})=0$ and
\(\|s_{k+2}\|=0\le \|s_{k+1}\|\). Hence,  the inequality again holds. Therefore, in both possibilities, we obtain $\|s_{k+2}\|\le \|s_{k+1}\|$ for every $k\ge0$, i.e.,
the sequence $\{\|s_k\|\}_{k\ge1}$ is nonincreasing. To prove the last statement in item (i), suppose that an interior step occurs at some index \(K\). Then \(p_{K+1}\in\Omega^*\) and \(\theta_{t_K}(p_K)=0\). Hence,  applying  Corollary~\ref{co:ctc}\,(ii)  we obtain that   
\(
\theta_{t_{K+1}}(p_{K+1})\,t_{K+1}\leq \theta_{t_K}(p_K)\,t_K=0, 
\)
which implies that \(\theta_{t_{K+1}}(p_{K+1})=0\). An induction argument then yields \(\theta_{t_j}(p_j)=0\) for all \(j\ge K+1\), which in turn gives \(s_{j+1}=\mathbf 0\) for all \(j\ge K+1\). This establishes item \textup{(i)}.

To prove item (ii), we first  apply  Theorem~\ref{thm:second-brox-Riem}\,(ii) with \((p,p_{+},x)=(p_k,p_{k+1},p_k)\) to obtain 
\[
f(p_k)-f(p_{k+1})
\ \ge\ 
\theta_{t_k}(p_k)\,\big\langle \log_{p_{k+1}}p_k,\ \log_{p_{k+1}}p_k\big\rangle
= \theta_{t_k}(p_k)\,d(p_k,p_{k+1})^2.
\]
By definitions of \(s_{k+1}\) and \(d\), we have
\(\|s_{k+1}\|=\theta_{t_k}(p_k)\,d(p_{k+1},p_k)\). Thus, it follows from the last inequality  that 
\[
t_k\,\|s_{k+1}\|= d(p_{k+1},p_k)\,\|s_{k+1}\|=\theta_{t_k}(p_k)\,d(p_{k+1},p_k)^2
\leq  f(p_k)-f(p_{k+1}).
\]
Since \(f(p_K)\ge f^*\),  summing for \(k=0,\ldots,K-1\) yields \eqref{eq:nonsmooth-sum}.  Now, if \(t_k\equiv t\), then dividing  \eqref{eq:nonsmooth-sum} by \(t\) and using  (i), we conclude that  \(\|s_{k+1}\|\ge \|s_{K}\|\) for all \(k\le K-1\). Thus, we have 
\[
K \|s_{K}\| \leq  \sum_{k=0}^{K-1}\|s_{k+1}\|=\frac{1}{t}\sum_{k=0}^{K-1} t\,\|s_{k+1}\|
\leq  \frac{f(p_0)-f^*}{t}.
\]
This is \eqref{eq:nonsmooth-last}, which concludes the proof. 
\end{proof}

We now state a geometric bound for the objective values along RB-PPM; for constant radii it yields a
linear rate up to the finite termination ensured by Theorem~\ref{thm:icfv}.

\begin{theorem}\label{th:linear}
For any \(K\ge 1\), we have 
\[
f(p_K)-f_*\leq  \Big[\prod_{k=0}^{K-1}\!\left(\frac{\operatorname{dist}(p_{0},\Omega^*)}{\operatorname{dist}(p_{0},\Omega^*)+{t_k}}\right)\Big]\bigl(f(p_0)-f_*\bigr).
\]
In particular, if \(t_k\equiv t>0\), then
\[
f(p_K)-f_* \leq \left(\frac{\operatorname{dist}(p_{0},\Omega^*)}{\operatorname{dist}(p_{0},\Omega^*)+{t}}\right)^{K}\!\bigl(f(p_0)-f_*\bigr),
\]
i.e., a geometric (linear–rate) decay holds up to the finite termination guaranteed by Theorem~\ref{thm:icfv}.
\end{theorem}

\begin{proof}
If some \(j\in\{0,\ldots,K-1\}\) satisfies \(p_{j+1}\in\Omega^*\), then \(f(p_K)=f_*\) and the bound is trivial. 
Otherwise, for each \(k=0,1,\ldots,K-1\)  the point \(p_{j+1}\) belongs to  the boundary of the ball. Thus, Lemma~\ref{le:RB-PPM-convex}\textup{(ii)}  implies that  \(\operatorname{dist}(p_{k+1},\Omega^*)\le \operatorname{dist}(p_0,\Omega^*)\), and using Lemma~\ref{eq:ppicb} we conclude that
\[
f(p_{k+1})-f_*\leq  \left(\frac{\operatorname{dist}(p_{0},\Omega^*)}{\operatorname{dist}(p_{0},\Omega^*)+{t_k}}\right)( f(p_k)-f_*), \qquad k=0,1,\ldots,K-1.
\]
Iterating this recursion for \(k=0,\ldots,K-1\) yields the claimed product bound. 
For the constant–radius case \(t_k\equiv t>0\), the product reduces to 
\(
\big({\operatorname{dist}(p_{0},\Omega^*)}/({\operatorname{dist}(p_{0},\Omega^*)+{t}})\big)^{K}, 
\)
establishing geometric decay until the method hits \(\Omega^*\)
in finitely many steps by Theorem~\ref{thm:icfv}.
\end{proof}

Before presenting a refined $\mathcal{O}(1/K)$ estimate, we record a simple product-free corollary
of Theorem~\ref{th:linear}. Although the geometric bound in Theorem~\ref{th:linear} already allows
one to estimate suitable values of $K$ and $t$, the expression below provides a closed-form
approximation that can be inverted directly with respect to these parameters.

\begin{remark}
Let $t>0$ be fixed and set $d_0:=\operatorname{dist}(p_0,\Omega^*)$. Using $(1+x)^K\ge 1+Kx$ for
$x\ge0$ (hence $(1+x)^{-K}\le(1+Kx)^{-1}$), we obtain
\[
\prod_{k=0}^{K-1}\frac{d_0}{d_0+t}
=\Bigl(1+\tfrac{t}{d_0}\Bigr)^{-K}
\le \frac{1}{\,1+\tfrac{K t}{d_0}\,}.
\]
Therefore, by Theorem~\ref{th:linear}, we have 
\[
f(p_K)-f_*\;\le\;\frac{d_0}{d_0+Kt}\,\bigl(f(p_0)-f_*\bigr).
\]
We state this bound because it is a direct, closed-form consequence of the geometric estimate and is
monotone in $K$, which is convenient for selecting the number of iterations or the radius. Its order
is $\mathcal{O}(1/K)$ and, in general, the constant appearing in the next theorem is sharper.
\end{remark}

The next result provides a refined $\mathcal{O}(1/K)$ last-iterate bound for RB-PPM with constant
radii, featuring an explicit constant that is typically tighter than the crude ${d_0}/{(d_0+Kt)}$
estimate above.

\begin{theorem}\label{thm:D3}
For any \(K\ge 1\), the iterates of RB-PPM with constant radii \(t_k\equiv t>0\) satisfy
\begin{equation} \label{eq:iclc}
f(p_K)-f^* \leq \min \left\{\frac{\operatorname{dist}(p_{0},\Omega^*)}{t} , ~\frac{\operatorname{dist}(p_{0},\Omega^*)^3}{(2\operatorname{dist}(p_{0},\Omega^*)+t)t^2}\right\}(f(p_0)-f^*)\frac{1}{K}.
\end{equation} 
\end{theorem}
\begin{proof}
Fix \(k\ge0\) with \(p_{k+1}\notin\Omega^*\), otherwise the claim is immediate.
By applying Theorem~\ref{thm:second-brox-Riem}\,(ii), for \(p=p_k\), \(p_{+}=p_{k+1}\), and any \({x}\in\mathcal M\), and also using  Cauchy–Schwarz,  we have 
\[
f(x)-f(p_{k+1})\geq \theta_t(p_k) \big\langle \log_{p_{k+1}}p_k,\ \log_{p_{k+1}}{x}\big\rangle\geq - \theta_t(p_k) \|\log_{p_{k+1}}p_k\| \|\log_{p_{k+1}}x\|.
\]
To simplify the notations,  set   \(d_k:=\operatorname{dist}(p_k,\Omega^*)\). Thus, choosing \(x=x_*\in\Omega^*\), taking into account that \(t_k\equiv t>0\),  \(\|\log_{p_{k+1}}p_k\|=d(p_{k+1},p_k)=t\) and \(\|\log_{p_{k+1}}p_*\|\geq d(p_{k+1},\Omega^*)\) yields
\begin{equation}\label{eq:gap-theta1}
0<f(p_{k+1})-f^*\leq  \theta_t(p_k) t d(p_{k+1},\Omega^*)=\theta_t(p_k)\,t\,d_{k+1}.
\end{equation}
On the other hand,  applying  Theorem~\ref{thm:second-brox-Riem}\,(ii) with $x=p_k$, $p_{+}=p_{k+1}$ and choosing  $p=p_k$ yields  
\[
f(p_k)-f(p_{k+1})
\geq \theta_t(p_k)\,\big\langle \log_{p_{k+1}}p_k,\ \log_{p_{k+1}}p_k\big\rangle
=\theta_t(p_k)\,\|\log_{p_{k+1}}p_k\|^2.
\]
Since $\|\log_{p_{k+1}}p_k\|=d(p_{k+1},p_k)=t$, we have 
\(
f(p_k)-f(p_{k+1}\geq \theta_t(p_k)\,t^2,
\)
which implies that 
\begin{equation}\label{eq:valuen}
\theta_t(p_k)\leq \frac{f(p_k)-f(p_{k+1})}{t^2}.
\end{equation}
Combining \eqref{eq:gap-theta1} with \eqref{eq:valuen} and taking into account that Lemma~\ref{le:RB-PPM-convex}
 implies $d_{k+1}\leq d_{0}$, we have
\begin{equation*}
0<f(p_{k+1})-f^*\leq    \frac{f(p_k)-f(p_{k+1})}{t}d_{k+1}\leq    \frac{f(p_k)-f(p_{k+1})}{t}d_{0}.
\end{equation*}
Summing this inequality for $k=0,\ldots,K-1$ and using the telescoping sums $\sum_{k=0}^{K-1}\!(f(p_k)-f(p_{k+1}))=f(p_0)-f(p_K)\le f(p_0)-f^*$, we obtain
\[
\sum_{k=0}^{K-1}\bigl(f(p_{k+1})-f^*\bigr) \leq  \frac{d_0}{t}\,\bigl(f(p_0)-f^*\bigr).
\]
Since Corollary~\ref{eq:monfv} implies  that $\{f(p_k)\}_{k\in \mathbb{N}}$ is nonincreasing, it follows from the last inequality  that 
\(
K\bigl(f(p_K)-f^*\bigr)\leq \frac{d_0}{t}\bigl(f(p_0)-f^*\bigr).
\)
Therefore,  we conclude that 
\begin{equation} \label{eq:feq}
f(p_K)-f^*\leq \frac{d_0(f(p_0)-f^*)}{t} \frac{1}{K},
\end{equation} 
Now,  we proceed  to obtain a new bound for $f(p_K)-f^*$.  For that, first note that, due to    \(p_{k+1}\notin\Omega^*\),  Theorem~\ref{thm:first-brox-Riem} implies that   $\mathbb B(p_k,t_k)\cap {\Omega^*}= \varnothing$. Thus, applying  Lemma~\ref{le:coslaw} with $p=p^*$ and considering that \(d(p_{k+1},p_k)=t\),  after  some algebraic manipulations,  we  obtain that 
\begin{equation*}
0<f(p_{k+1})-f({p^*}) +  \frac{\theta_{t_k}(p_k)}{2}t^2\leq  \frac{\theta_{t_k}(p_k)}{2}(d^2(p_k,{p^*}) - d^2(p_{k+1},{p^*})). 
\end{equation*}
Thus, using \eqref{eq:valuen} and the inequalities  \(f(p_k)-f(p_{k+1})\le f(p_0)-f^*\) and  \(d^2(p_{k+1},{p}_*) \leq d^2(p_k,{p}_*)\), we conclude  from the  last inequality that 
\begin{equation} \label{eq:coslawap2}
0<f(p_{k+1})-f^* +  \frac{\theta_{t_k}(p_k)}{2}t^2\leq  \frac{f(p_0)-f^*}{2t^2}(d^2(p_k,{p^*}) - d^2(p_{k+1},{p^*})). 
\end{equation}
Summing this  inequality for $k=0,\ldots,K-1$  and noting that the right-hand side
telescopes, we obtain
\[
\sum_{k=0}^{K-1}\!\big(f(p_{k+1})-f^*\big)+\frac{t^2}{2}\sum_{k=0}^{K-1}\theta_{t_k}(p_k) \leq \frac{f(p_0)-f^*}{2t^2}\sum_{k=0}^{K-1}\!\big(d^2(p_k,p^*)-d^2(p_{k+1},p^*)\big)\leq \frac{f(p_0)-f^*}{2t^2}\,d_0^2, 
\]
since \eqref{eq:coslawap2} holds for all $p^*\in \Omega^*$. Thus, considering that  Corollary~\ref{eq:monfv} implies  that $\{f(p_k)\}_{k\in \mathbb{N}}$ is nonincreasing, we conclude that 
\[
K\big(f(p_K)-f^*\big)\leq \sum_{k=0}^{K-1}\!\big(f(p_{k+1})-f^*\big).
\]
Therefore, combining the two previous  inequalities we conclude that 
\begin{equation} \label{eq:qfinq}
(f(p_K)-f^*)  +\frac{1}{K} \frac{t^2}{2}\sum_{k=0}^{K-1}\theta_{t_k}(p_k) \leq \frac{d_0^2(f(p_0)-f^*)}{2t^2} \frac{1}{K}. 
\end{equation} 
Next, we obtain a lower bound for the second term on the left-hand side of \eqref{eq:qfinq}. First, note that  Lemma~\ref{le:sqthek}  implies that  $\{\theta_t(p_k)\}_{k=0}^{K-1}$ is nonincreasing. Thus, 
\begin{equation}\label{eq:avg-theta-start}
\frac{t^{2}}{K}\sum_{k=0}^{K-1}\theta_t(p_k)\geq t^{2}\,\theta_t(p_{K-1}).
\end{equation}
On the other hand, by applying  Corollary~\ref{co:ctc}(i)   with the points  \(p=p_{K-1}\), \(p_{+}=p_K\)  we obtain that 
\begin{equation}\label{eq:theta-lb-last}
\theta_t(p_{K-1})\geq \frac{f(p_K)-f^*}{\,d(p_{K-1},p_K)\,d(p_K,p^*)\,}.
\end{equation}
Multiplying \eqref{eq:theta-lb-last} by $t^2$ and using the boundary identity $d(p_{K-1},p_K)=t$  we concludes that 
\begin{equation}\label{eq:avg-theta-lb-riem}
t^{2}\,\theta_t(p_{K-1})\geq  t\,\frac{f(p_K)-f^*}{\,d(p_K,p^*)\,}.
\end{equation}
Finally, Lemma~\ref{le:RB-PPM-convex}(ii) also implies
$d(p_K,p^*)\le d(p_0,p^*)$. Thus,   combining \eqref{eq:avg-theta-start} and
\eqref{eq:avg-theta-lb-riem} gives
\begin{equation*} \label{eq:fexpb}
\frac{t^{2}}{K}\sum_{k=0}^{K-1}\theta_t(p_k)
\geq t\,\frac{f(p_K)-f^*}{\,d(p_0,p^*)\,}.
\end{equation*}
Combining \eqref{eq:qfinq}  with the preceding inequality, valid for all \(p^*\in\Omega^*\), and rearranging terms, we obtain 
\begin{equation*} 
f(p_K)-f^* \leq \frac{\operatorname{dist}(p_{0},\Omega^*)^3\bigl(f(p_0)-f^*\bigr)}{(2\operatorname{dist}(p_{0},\Omega^*)+t)t^2}\frac{1}{K}.
\end{equation*} 
Therefore, the last inequality together with \eqref{eq:feq} yields \eqref{eq:iclc}. This completes the proof.
\end{proof}

Theorems~\ref{th:linear} and~\ref{thm:D3} provide two complementary types of nonasymptotic guarantees for RB--PPM with constant radii, namely,  a geometric decay of the function values and an explicit $\mathcal{O}(1/K)$ last–iterate bound. To clarify their respective roles, and the dependence of the associated constants on the initial distance and the radius, we briefly compare these results below.

\begin{remark}
Let us make a comparison between Theorems~\ref{th:linear} and~\ref{thm:D3}. For that,   $d_0:=\operatorname{dist}(p_0,\Omega^*)$, $\Delta_0:=f(p_0)-f^*>0$ and  constant radii $t_k\equiv t>0$.\\
\noindent
\textit{\bf (a) Nature of the rates.} Using this notations, Theorem~\ref{th:linear} yields, for every $K\ge1$,
\begin{equation} \label{eq:igr}
f(p_K)-f^* \leq q(d_0,t)^K \Delta_0, \qquad \quad   q(d_0,t):=\frac{d_0}{d_0+t}\in(0,1).
\end{equation} 
that is, a geometric (linear--rate) decay with contraction factor
\[
q(d_0,t):=\frac{d_0}{d_0+t}
=\frac{1}{1+t/d_0}\in(0,1).
\]
In contrast, Theorem~\ref{thm:D3} provides a last--iterate sublinear estimate of the form
\[
f(p_K)-f^*\leq  C(d_0,t)\,\frac{1}{K}\,\Delta_0, \qquad  \quad C(d_0,t):=\min\left\{\frac{d_0}{t},\ \frac{d_0^3}{(2d_0+t)t^2}\right\}.
\]
Thus, Theorem~\ref{th:linear} describes a purely geometric contraction, whereas Theorem~\ref{thm:D3} gives an explicit $\mathcal{O}(1/K)$ bound for the last iterate with a computable constant.\\
\noindent
\textit{\bf (b) Dependence on $d_0$ and $t$.}
Both results depend on the initial distance $d_0$ and on the radius $t$, but in different ways.

 For Theorem~\ref{th:linear}, note that  the factor $q(d_0,t)$ satisfies the following equality 
\[
q(d_0,t)=\frac{d_0}{d_0+t}=1-\frac{t}{d_0+t}.
\]
Thus, for fixed $d_0$,  we conclude that for  \(t\ll d_0\)  we have a slow contraction in \eqref{eq:igr}  due to 
\[
q(d_0,t)\approx 1-\frac{t}{d_0}
\]
while $t$ is comparable with $d_0$ forces $q(d_0,t)$ to be uniformly bounded away from $1$, leading to a faster reduction of $f(p_K)-f^*$.

For Theorem~\ref{thm:D3}, the constant $C(d_0,t)$ also depends on the ratio $t/d_0$. In particular, we can rewrite
\[
C(d_0,t)=\min\left\{\frac{1}{t/d_0},\ \frac{1}{(2+t/d_0)(t/d_0)^2}\right\}=\begin{cases}
\dfrac{d_0}{t}, & 0<t/d_0<\sqrt2-1,\\[1.5ex]
\dfrac{d_0^3}{(2d_0+t)t^2}, & t/d_0\ge\sqrt2-1.
\end{cases}
\]
and one checks that the two terms coincide at $t/d_0=\sqrt2-1$.
Hence, for relatively small radii $t<(\sqrt2-1)d_0$ the first term $d_0/t$ governs the constant, whereas for $t$ comparable to or larger than $d_0$ the second term becomes smaller and yields a sharper $\mathcal{O}(1/K)$ constant. In all cases, starting farther from $\Omega^*$ (larger $d_0$) deteriorates both the geometric factor $q(d_0,t)$ and the constant $C(d_0,t)$.

\smallskip

\noindent
\textit{\bf (c) Complexity interpretation.}
Fix $d_0$, $t$ and a target accuracy $\varepsilon\in(0,\Delta_0)$. From Theorem~\ref{th:linear}, the inequality
\[
\left(\frac{d_0}{d_0+t}\right)^K\Delta_0\le\varepsilon  \ \Longrightarrow\ K\ \ge\ \frac{d_0+t}{t}\,\log\!\left(\frac{\Delta_0}{\varepsilon}\right),
\]
which is the usual logarithmic iteration complexity associated with a linear rate, with a constant factor $(d_0+t)/t$. On the other hand, Theorem~\ref{thm:D3} gives the polynomial requirement
\[
C(d_0,t)\,\frac{1}{K}\,\Delta_0\le\varepsilon
\ \Longrightarrow\
K\ \ge\ \frac{C(d_0,t)\,\Delta_0}{\varepsilon}.
\]
Thus, for fixed $(d_0,t)$ and sufficiently large $K$, the geometric decay in Theorem~\ref{th:linear} is asymptotically stronger than the $\mathcal{O}(1/K)$ estimate of Theorem~\ref{thm:D3}. However, Theorem~\ref{thm:icfv} shows that RB--PPM with constant radius terminates after at most $\lceil d_0^2/t^2\rceil$ iterations, so both bounds should be interpreted as nonasymptotic complexity estimates valid up to the finite hitting time. From this perspective, Theorem~\ref{th:linear} provides a compact product (and, in the constant--radius case, geometric) description of the decay, while Theorem~\ref{thm:D3} complements it with an explicit $\mathcal{O}(1/K)$ last--iterate bound whose constant can be analyzed and optimized in terms of the ratio $t/d_0$.
\end{remark}

%%%%%%%%%%%%%%%%%%%%%%%%%
\subsection{Asymptotic convergence analysis}  \label{sec:AsympAnalysis}
In this section we study the asymptotic behaviour of the sequence generated by RB--PPM under the mild global condition
\(\sum_{k=0}^{+\infty} t_k = +\infty\), which is the natural analogue of the ``unbounded total step size'' assumption
in classical proximal point methods. We show that, in the absence of minimizers, the sequence of values necessarily
converges  the infimum, wheres in the presence of a nonempty solution set the iterates themselves converge to a minimizer.
Thus, Theorem~\ref{thm:dichotomy} provides a complete dichotomy between the cases \(\Omega^*=\varnothing\) and \(\Omega^*\neq\varnothing\).

\begin{theorem} \label{thm:dichotomy}
Let $\{p_k\}_{k\in\mathbb N}$ be the RB-PPM sequence with sequence  of radii $\{t_k\}_{k\ge0}$. Suppose that $\sum_{k=0}^{+\infty} t_k=+\infty$. Then $\lim_{k\to\infty} f(p_k)=f^*$. In addition, 
 if $\Omega^*\neq \varnothing$, then
\(
\lim_{k\to\infty} p_k=p_*\in \Omega^*.
\)
\end{theorem}

\begin{proof}
First, assume that  $\Omega^*=\varnothing$. In this case,  Theorem~\ref{thm:second-brox-Riem} implies that  $\{p_k\}_{k\in\mathbb N}\cap \Omega^*=\varnothing$, $d(p_{k+1},p_k)=t_k$ and \( \{\theta_{t_k}(p_k)\}_{k\in\mathbb{N}} \subset \mathbb{R}_{++}\). In addition,  Lemma~\ref{le:coslaw} implies that 
\begin{equation} \label{eq:coslawapac}
d^2(p_{k+1},{p}) \leq d^2(p_k,{p}) - d^2(p_k,p_{k+1})+ \frac{2}{\theta_{t_k}(p_k)}(f({p})-f(p_{k+1})),  \qquad \forall p\in\mathcal{M},\quad   \forall k\in \mathbb{N}.
\end{equation} 
Since $\{p_k\}_{k\in\mathbb N}\cap \Omega^*=\varnothing$ and  \( \{\theta_{t_k}(p_k)\}_{k\in\mathbb{N}} \subset \mathbb{R}_{++}\), substituting $p=p_k$ in the last inequality we conclude that $\{f(p_k)\}_{k\in\mathbb N}$ is  decreasing.  Assume by contradiction that $\lim_{k\to\infty} f(p_k)>f^*$. Due  to the sequence    $\{f(p_k)\}_{k\in\mathbb N}$ be  decreasing and ${f^*}:=\inf_{p\in\mathcal{M}} f(p)$, there exists ${\bar p}\in \mathcal{M}$ and  $\epsilon>0$, such that $f^*<f({\bar p})<f(p_k)-\epsilon$, for all $k\in\mathbb{N}$. Substituting $p={\bar p}$ into \eqref{eq:coslawapac}, after some algebraic manipulations,  we conclude that 
\begin{equation} \label{eq:coslawapac1}
\frac{1}{\theta_{t_k}(p_k)} \leq \frac{1}{2\epsilon} \big(d^2(p_k,{\bar p}) - d^2(p_{k+1},{\bar p})\big),  \qquad   \forall k\in \mathbb{N}.
\end{equation}
It follows from Lemma~\ref{le:sqthek} that the sequence  \( \{t_k\theta_{t_k}(p_k)\}_{k\in\mathbb{N}} \) is nonincreasing. Hence, we conclude that 
\(t_{k} {\theta_{t_k}(p_k)}\leq {t_0\theta_{t_0}(p_0)}\), for all \(k\in \mathbb{N}\). Thus, it follows from \eqref{eq:coslawapac1} that 
\begin{equation} \label{eq:coslawapac2}
t_k  \leq \frac{{t_0\theta_{t_0}(p_0)}}{2\epsilon} \big(d^2(p_k,{\bar p}) - d^2(p_{k+1},{\bar p})\big),  \qquad   \forall k\in \mathbb{N}.
\end{equation}
Summing \eqref{eq:coslawapac2} from $k=0$ to $N-1$, we obtain
\[
\sum_{k=0}^{N-1}t_k  \leq \frac{{t_0\theta_{t_0}(p_0)}}{2\epsilon} \sum_{k=0}^{N-1}\big(d^2(p_k,{\bar p}) - d^2(p_{k+1},{\bar p})\big)< \frac{{t_0\theta_{t_0}(p_0)}}{2\epsilon} d^2(p_0,{\bar p})<+\infty, \qquad \forall N\in \mathbb{N}, 
\]
which contradicts the equality $\sum_{k=0}^{+\infty} t_k=+\infty$. Therefore, if  $\Omega^*=\varnothing$, then $\lim_{k\to\infty} f(p_k)=f^*$.

Now assume that $\Omega^*\neq\varnothing$. For simplifying  we denote \(d_k:=\operatorname{dist}(p_k,\Omega^*)\).  By Lemma~\ref{le:RB-PPM-convex}(i)-(ii),  $\mathrm{dist}(p_{k+1},\Omega^*)\le \mathrm{dist}(p_k,\Omega^*)\le d_0$. Hence, it follows from Lemma~\ref{eq:ppicb}  that 
\[
f(p_{k+1})-f^*\leq  \frac{d_0}{d_0+t_k}\,\bigl(f(p_k)-f^*\bigr), \qquad  \forall k\in \mathbb{N}.
\]
Some calculations show that, for a fixed $K\in \mathbb{N}$,   the last inequality implies that 
\begin{equation}\label{eq:prod}
f(p_K)-f^*\leq  \bigl(f(p_0)-f^*\bigr)\,\prod_{k=0}^{K-1}\frac{d_0}{d_0+t_k} = \bigl(f(p_0)-f^*\bigr)e^{-\sum_{k=0}^{K-1}\ln\!\big(1+\tfrac{t_k}{d_0}\big)}.
\end{equation}
We claim that $\sum_{k=0}^{+\infty} \ln(1+t_k/d_0)=+\infty$ whenever $\sum_{k=0}^\infty t_k=+\infty$.
Indeed, split the index set into
\[
\mathcal I_1:=\{k:\ t_k\le d_0\},\qquad \mathcal I_2:=\{k:\ t_k>d_0\}.
\]
For $k\in\mathcal I_1$,   i.e.,   for $t_k\le d_0$,  we have the bound $\ln(1+\tfrac{t_k}{d_0})\geq \tfrac{t_k}{2d_0}$. For $k\in\mathcal I_2$, we have the bound  $\ln(1+t_k/d_0)\ge \ln 2$.
Hence,
\[
\sum_{k=0}^{K-1}\ln\!\big(1+\tfrac{t_k}{d_0}\big)
\geq  \frac{1}{2d_0}\sum_{k\in\mathcal I_1\cap[0:K-1]} t_k+(\log 2)\,|\mathcal I_2\cap[0:K-1]|, 
\]
where \(,|\mathcal I_i\cap[0:K-1]|:= \#\{\,k\in\{0,1,\ldots,K-1\}:\ k\in \mathcal I_i\,\}\), for $i=1,2$. Hence, if $\sum_{k=0}^\infty t_k=+\infty$, then either $\sum_{k\in\mathcal I_1}t_k=+\infty$ and the first term diverges or $\mathcal I_2$ is infinite  and the second term diverges. In both cases the sum diverges, so the product in \eqref{eq:prod} tends to zero and we conclude
\begin{equation} \label{eq:ctfs}
\lim_{k\to\infty}\bigl(f(p_k)-f^*\bigr)=0.
\end{equation} 
Since $\Omega^*\neq\varnothing$, fix ${\bar p}\in\Omega^*$. 
Because $\{f(p_k)\}_{k\in\mathbb N}$ is decreasing and $f({\bar p})=f^*$, we conclude that $f({\bar p})\le f(p_k)$,  for all $k\in\mathbb N$.
Substituting $p={\bar p}$ into \eqref{eq:coslawapac} yields
\[
d^{2}(p_{k+1},{\bar p})\leq  d^{2}(p_k,{\bar p}), \qquad  \forall k\in \mathbb{N}.
\]
Hence, by Definition~\ref{def:QuasiFejer}, the sequence $\{p_k\}_{k\in\mathbb N}$ is Fejér convergent to $\Omega^*$.  Since $\Omega^*\neq\varnothing$, it follows from Theorem~\ref{teo.qf} that  $\{p_k\}_{k\in\mathbb N}$   is bounded. Considering that  \(\mathcal M\) is complete, $\{p_k\}_{k\in\mathbb N}$  admits a convergent subsequence. Let  $\{p_{k_j}\}_{j\in\mathbb N}$  be a subsequence with \(\lim_{j\to\infty} p_{k_j}= {\hat  p}\).
Because \(f\) is  lower semicontinuous continuous we have  \(\lim_{j\to\infty} f(x_j)=f( {\hat  p})\). Moreover, because  $\{f(p_k)\}_{k\in\mathbb N}$ be  decreasing we conclude that the whole sequence converges, i.e., $\lim_{k\to\infty} f(p_k)=f( {\hat  p})$, which combined with \eqref{eq:ctfs} implies that \({\hat  p}\in \Omega^*\). Therefore, Fejér convergnece to the nonempty closed set $\Omega^*$ together with the existence of a cluster point \(\hat p\in\Omega^*\) implies, by Theorem~\ref{teo.qf}, that the whole sequence converges, i.e., \(\lim_{k\to\infty} p_k=p_*\in \Omega^*\), which proves the theorem.
\end{proof}

\begin{corollary}\label{cor:asym-constant}
If \(t_k\equiv t>0\), then the conclusions of Theorem~\ref{thm:dichotomy} hold. In particular,
\(
\lim_{k\to\infty} f(p_k) = f^*
\)
and, if \(\Omega^*\neq\varnothing\), there exists \(p_*\in\Omega^*\) such that
\(\lim_{k\to\infty} p_k = p_*\).
\end{corollary}
\begin{proof}
Since \(\sum_{k=0}^{+\infty} t_k = +\infty\) whenever \(t_k \equiv t > 0\), the result  follows directly from Theorem~\ref{thm:dichotomy}.
\end{proof}

\begin{remark}
Theorem~\ref{thm:dichotomy} complements the nonasymptotic results of Section~\ref{sec:IntComplAnalysis} in two directions.
First, in the constant--radius case \(t_k\equiv t>0\), Theorem~\ref{thm:icfv} guarantees finite termination after at most
\(\lceil \operatorname{dist}^2(p_0,\Omega^*)/t^2\rceil\) steps when \(\Omega^*\neq\varnothing\); in this regime, the convergence
statements of Theorem~\ref{thm:dichotomy} are trivially satisfied. Second, when the radii are allowed to vary but still satisfy
\(\sum_{k=0}^{+\infty} t_k = +\infty\), Theorem~\ref{thm:dichotomy} ensures that the global behaviour of RB--PPM remains well controlled, namely, 
the function values always converge to the infimum, and the iterates converge to a minimizer whenever one exists. Combined with the
linear and $\mathcal{O}(1/K)$ bounds in Theorems~\ref{th:linear} and~\ref{thm:D3}, this yields a complete picture of both the global
and nonasymptotic behaviour of RB--PPM.
\end{remark}

%%%%%%%%%%%%%%%%%%%%%%%%%%%%
\section{Numerical experiments}\label{Sec:NumExp}

This section presents numerical experiments on a classical and controlled benchmark problem, designed to illustrate the practical behavior of ball-proximal schemes. Although Euclidean ball-proximal methods were introduced in \cite{gruntkowska2025,gruntkowska2025n}, the existing literature is predominantly theoretical and offers only limited numerical illustration. The experiments below are therefore intended as an initial comparative computational study of the broximal approach, highlighting its practical behavior relative to proximal and gradient methods. The experiments were conducted in MATLAB version R2024b (24.2.0.2712019), on a computer with a 3.7 GHz Intel Core i5 6-Core processor and 8 GB 2667 MHz DDR4 RAM, running macOS Sequoia 15.7.2. The codes are available at \url{https://github.com/lfprudente/Broximal}.

Although the theory developed in this paper is stated in the general setting of Hadamard manifold, we consider here the simplest Euclidean test case, namely,  the minimization of a strongly convex quadratic function in $\mathbb{R}^n$. Specifically, we consider
\begin{equation}\label{eq:quadratic}
    \min_{x\in\mathbb{R}^n} f(x) := \frac{1}{2}x^\top A x,
\end{equation}
where $A\in\mathbb{R}^{n\times n}$ is symmetric positive definite.   This elementary and controlled setting allows us to isolate the algorithmic behavior of the methods without additional geometric or modeling effects. The resulting observations provide a first indication of the practical behavior of ball-proximal schemes, including the influence of the radius selection strategy, and may serve as a useful reference for future studies on more general, more challenging, and intrinsically non-Euclidean problems. The comparison with proximal and gradient methods serves primarily as a reference, allowing us to assess the broximal schemes against methods whose behavior on this class of problems is already well known. For comparison, we consider four algorithms:

\begin{itemize}
    \item {\bf Broximal-F:} RB-PPM with {\it fixed radius} $t_k\equiv t$;
    \item {\bf Broximal-A:} RB-PPM with {\it adaptive radius}
    \(
        t_k=\min\{t_{\max},\max\{t_{\min},\alpha\|\nabla f(x_k)\|\}\},
    \)
    where $\alpha$, $t_{\min}$, and $t_{\max}$ are algorithmic parameters;
    \item {\bf Broximal-P:} RB-PPM with {\it Polyak-type radius}
    \(
        t_k= \min\left\{ t_{\max},  \max\left\{t_{\min},\beta\tfrac{f(x_k)-f^*}{\|\nabla f(x_k)\|}\right\}\right\}
    \),
    where $\beta$, $t_{\min}$, and $t_{\max}$ are algorithmic parameters;
    \item {\bf Proximal:} the classical proximal point algorithm \eqref{eq:ppm} with $\lambda_k\equiv \lambda$;
    \item {\bf Gradient:} the gradient method with Armijo line search applied directly to the original problem.
\end{itemize}

The {\it adaptive radius choice} is motivated by the descent lemma for geodesically \(L\)-smooth functions and by the fact that, at every nonstationary iteration of RB--PPM, one has \(d(x_k,x_{k+1})=t_k.\) Thus, the radius is not merely a feasibility parameter; it is exactly the length of the displacement produced by the method. To make this connection precise, define
\[
\tau_k:=\min\Bigl\{\frac{t_k}{\|\grad f(x_k)\|},\frac1L\Bigr\},
\qquad
y_k:={\rm exp}_{x_k}\!\bigl(-\tau_k \grad f(x_k)\bigr).
\]
Then, \(d(x_k,y_k)=\tau_k\|\grad f(x_k)\|\le t_k,\) whivh implies that  \(y_k\in {B}[x_k,t_k]\). Since \(x_{k+1}\) minimizes \(f\) over \({B}[x_k,t_k]\), we have
\(f(x_{k+1})\le f(y_k).\) Using the descent lemma along the geodesic gradient step, we obtain
\[
f(x_k)-f(x_{k+1})
\ge
\tau_k\Bigl(1-\frac{L\tau_k}{2}\Bigr)\|\grad f(x_k)\|^2.
\]
Equivalently,
\[
f(x_k)-f(x_{k+1})
\ge
\begin{cases}
t_k\|\grad f(x_k)\|-\dfrac{L}{2}t_k^2,
& \text{if } t_k\le \dfrac{\|\grad f(x_k)\|}{L},\\[0.8em]
\dfrac{1}{2L}\|\grad f(x_k)\|^2,
& \text{if } t_k\ge \dfrac{\|\grad f(x_k)\|}{L}.
\end{cases}
\]
Therefore, when \(t_k\le \|\grad f(x_k)\|/L\), the guaranteed decrease is a concave quadratic function of the step length \(t_k\). In particular, from the viewpoint of this lower bound, enlarging the radius is useful only up to the natural scale \(\|\grad f(x_k)\|/L\). Beyond this threshold, the guaranteed decrease no longer improves, while larger radii may make the ball-constrained subproblem more expensive to solve. This suggests choosing the parameter \(\alpha\) in the adaptive rule on the scale of the inverse local Lipschitz constant.

 A {\it second natural choice is the Polyak-type radius}, which is the ball-radius analogue of the classical Polyak stepsize. Indeed, in subgradient-type  methods the Polyak rule scales the step length by the ratio between the current objective gap and the norm of a first-order residual. In the present context, since the radius plays the role of a step-length parameter and exact boundary steps satisfy \(d(x_{k+1},x_k)=t_k\), this choice provides a natural way to adapt the displacement to both function values and first-order information. 
In particular, under the convexity assumption, for any \(p^\ast\in\Omega^\ast\) one has\(f(x_k)-f^\ast \le \|\nabla f(x_k)\|\,d(x_k,p^\ast),\) and therefore
\[
0<\tfrac{f(x_k)-f^\ast}{\|\nabla f(x_k)\|}\le \operatorname{dist}(x_k,\Omega^\ast).
\]
Hence, near the solution set the quotient \((f(x_k)-f^\ast)/\|\nabla f(x_k)\|\) is automatically small, leading to smaller and more conservative radii.

All methods were stopped when
\[
    \|\nabla f(x_k)\|\le \varepsilon_{\rm opt}, \qquad \varepsilon_{\rm opt}=10^{-6}.
\]
Several values were tested for each algorithmic parameter. The representative values considered in the experiments are
\begin{equation} \label{eq;pchoice}
\small
  t\in\{0.05,0.10,0.50\},\quad \alpha\in\{10^{-4},10^{-3},10^{-2}\},\quad \beta\in\{10^{-3},10^{-2},10^{-1}\},\quad \lambda\in\{10^{-3},10^{-1},2\}.
\end{equation} 
Here, $t$ is the fixed radius in Broximal-F, $\alpha$ is the scaling parameter in Broximal-A, $\beta$ is the scaling parameter in Broximal-P and $\lambda$ is the regularization parameter in the Proximal algorithm. For both Broximal-A and  Broximal-P, we used $t_{\min}=\varepsilon_{\rm opt}^{1/3}$ and $t_{\max}=10^6$.

The theoretical RB--PPM studied in this paper is an exact scheme, namely,  at each outer iteration, the next iterate is defined as a minimizer of $f$ over the ball $B(x_k,t_k)$. Since solving this subproblem exactly is generally impractical, we consider an inexact implementation. {\it For all ball-proximal variants, the ball subproblems are approximately solved by a  projected gradient method with Armijo line search}. The inner solver is stopped by a first-order residual criterion with an adaptive tolerance. This tolerance is initialized at $\sqrt{\varepsilon_{\rm opt}}$, bounded below by $\varepsilon_{\rm opt}$, and progressively reduced as the outer optimality residual decreases. For the ball-proximal subproblems, the tolerance is additionally bounded above by $0.1t_k$, so that the accuracy of the inner solution is compatible with the current ball radius. For a fair comparison, the {\it proximal method was implemented in a comparable outer-inner fashion, with each proximal subproblem solved approximately by the gradient method with Armijo line search}. The same line search parameters were used throughout, namely Armijo parameter \(10^{-4}\) and backtracking factor \(1/2\).

For each dimension $n\in\{100,200,\ldots,1000\}$, the matrix $A$ in \eqref{eq:quadratic} was generated as
\[
    A = UDU^\top,
\]
where \(D=\operatorname{diag}(d_1,\ldots,d_n)\), with \(d_i\) equally spaced in \([1,1000]\), and \(U\) is an orthogonal matrix obtained from a random matrix. 
Observe that, in this case, \(f\) has Lipschitz constant \(L=10^{3}\), which provides a rationale for the choice of the algorithmic parameters stated in \eqref{eq;pchoice}. The initial point \(x^0\) was generated with independent components uniformly distributed in \([-3,3]\).

Table~\ref{tab:results} reports the numerical results. In the table, $n$ denotes the number of variables, \emph{Alg.} identifies the algorithm, and \emph{Param.} gives the corresponding algorithmic parameter, namely,  the fixed radius $t$ for Broximal-F, the scaling parameter $\alpha$ for Broximal-A, the scaling parameter $\beta$ for Broximal-P and the proximal parameter $\lambda$ for Proximal. The column \emph{Outer(Inner)} reports the number of outer iterations and, in parentheses, the total number of inner iterations used to solve the subproblems. The column \#$f$ gives the total number of function evaluations, and \emph{time} gives the CPU time in seconds. The smallest reported values of \#$f$ and of CPU time for each dimension are highlighted in bold. We also note that all algorithms solved all instances within the prescribed tolerance. The typical final objective value $f(x^*)$ was of order $\mathcal{O}(10^{-13})$, while the typical final optimality residual $\|\nabla f(x^*)\|$ was of order $\mathcal{O}(10^{-7})$, where $x^*$ denotes the solution returned by the corresponding method. Since all methods produced comparable final objective values and optimality residuals, these quantities were omitted from the table to improve readability.

\begin{table}[h!]
\centering
\tiny
\begin{minipage}[t]{0.45\textwidth}
\begin{tabular}{|c|ccr@{\hspace{0.1em}}ccc|}
\hline
$n$ & Alg. & Param. & \multicolumn{2}{c}{Outer(Inner)} & \#$f$ & time \\ 
\hline

 &  & 0.05 & 439 & (6910) & 61471 & 0.84 \\ 
 &  & 0.10 & 221 & (5897) & 53009 & 0.72 \\ 
 & \multirowcell{-3}{Broximal-F} & 0.50 & 46 & (5477) & 50537 & 0.71 \\ 
\hhline{*1{~}|*6{-}|}
 &  & $10^{-4}$ & 501 & (3320) & \bf 25098 & \bf 0.35 \\ 
 &  & $10^{-3}$ & 167 & (5241) & 44887 & 0.59 \\ 
 & \multirowcell{-3}{Broximal-A} & $10^{-2}$ & 4 & (5057) & 50333 & 0.65 \\ 
\hhline{*1{~}|*6{-}|}
 &  & $10^{-3}$ & 2181 & (8504) & 62981 & 0.82 \\ 
 &  & $10^{-2}$ & 979 & (3972) & 29342 & 0.39 \\ 
 & \multirowcell{-3}{Broximal-P} & $10^{-1}$ & 150 & (5424) & 49120 & 0.60 \\ 
\hhline{*1{~}|*6{-}|}
 &  & $10^{-3}$ & 5 & (5435) & 54138 & 0.69 \\ 
 &  & $10^{-1}$ & 7 & (7400) & 73709 & 0.94 \\ 
 & \multirowcell{-3}{Proximal} & $2$ & 25 & (22680) & 225889 & 2.96 \\ 
\hhline{*1{~}|*6{-}|}
\multirowcell{-13}{100} & Gradient & -- & \multicolumn{2}{c}{5054} & 50343 & 0.83 \\ 
\hline\hline

 &  & 0.05 & 595 & (8988) & 75434 & 1.43 \\ 
 &  & 0.10 & 299 & (11313) & 100444 & 1.89 \\ 
 & \multirowcell{-3}{Broximal-F} & 0.50 & 62 & (10054) & 89097 & 1.72 \\ 
\hhline{*1{~}|*6{-}|}
 &  & $10^{-4}$ & 638 & (8810) & 73259 & 1.35 \\ 
 &  & $10^{-3}$ & 239 & (5463) & \bf 43146 & \bf 0.83 \\ 
 & \multirowcell{-3}{Broximal-A} & $10^{-2}$ & 4 & (6581) & 65501 & 1.26 \\ 
\hhline{*1{~}|*6{-}|}
 &  & $10^{-3}$ & 2964 & (9962) & 66362 & 1.34 \\ 
 &  & $10^{-2}$ & 1083 & (7326) & 54943 & 1.28 \\ 
 & \multirowcell{-3}{Broximal-P} & $10^{-1}$ & 179 & (9903) & 83480 & 1.63 \\ 
\hhline{*1{~}|*6{-}|}
 &  & $10^{-3}$ & 5 & (7224) & 71952 & 1.52 \\ 
 &  & $10^{-1}$ & 7 & (11050) & 110051 & 3.04 \\ 
 & \multirowcell{-3}{Proximal} & $2$ & 33 & (41167) & 409981 & 10.68 \\ 
\hhline{*1{~}|*6{-}|}
\multirowcell{-13}{200} & Gradient & -- & \multicolumn{2}{c}{6593} & 65667 & 2.69 \\ 
\hline\hline

 &  & 0.05 & 749 & (9500) & 79478 & 2.20 \\ 
 &  & 0.10 & 376 & (10095) & 88069 & 2.25 \\ 
 & \multirowcell{-3}{Broximal-F} & 0.50 & 77 & (10323) & 83299 & 2.44 \\ 
\hhline{*1{~}|*6{-}|}
 &  & $10^{-4}$ & 603 & (7993) & 64345 & 1.81 \\ 
 &  & $10^{-3}$ & 232 & (6463) & \bf 51166 & \bf 1.63 \\ 
 & \multirowcell{-3}{Broximal-A} & $10^{-2}$ & 4 & (6836) & 68043 & 1.95 \\ 
\hhline{*1{~}|*6{-}|}
 &  & $10^{-3}$ & 3638 & (11549) & 73689 & 1.85 \\ 
 &  & $10^{-2}$ & 1128 & (8822) & 69054 & 1.67 \\ 
 & \multirowcell{-3}{Broximal-P} & $10^{-1}$ & 184 & (8348) & 64453 & 1.61 \\ 
\hhline{*1{~}|*6{-}|}
 &  & $10^{-3}$ & 5 & (7591) & 75604 & 2.27 \\ 
 &  & $10^{-1}$ & 7 & (11717) & 116689 & 3.77 \\ 
 & \multirowcell{-3}{Proximal} & $2$ & 34 & (44327) & 441442 & 18.67 \\ 
\hhline{*1{~}|*6{-}|}
\multirowcell{-13}{300} & Gradient & -- & \multicolumn{2}{c}{6841} & 68135 & 3.12 \\ 
\hline\hline

 &  & 0.05 & 902 & (6450) & 49644 & 1.64 \\ 
 &  & 0.10 & 452 & (7860) & 64551 & 2.15 \\ 
 & \multirowcell{-3}{Broximal-F} & 0.50 & 92 & (8115) & 69066 & 2.22 \\ 
\hhline{*1{~}|*6{-}|}
 &  & $10^{-4}$ & 845 & (8691) & 68528 & 2.27 \\ 
 &  & $10^{-3}$ & 377 & (5967) & 46080 & 1.88 \\ 
 & \multirowcell{-3}{Broximal-A} & $10^{-2}$ & 4 & (6849) & 68174 & 2.21 \\ 
\hhline{*1{~}|*6{-}|}
 &  & $10^{-3}$ & 4265 & (12119) & 74018 & 2.39 \\ 
 &  & $10^{-2}$ & 1288 & (5903) & 37440 & 1.10 \\ 
 & \multirowcell{-3}{Broximal-P} & $10^{-1}$ & 187 & (4678) & \bf 33843 & \bf 0.93 \\ 
\hhline{*1{~}|*6{-}|}
 &  & $10^{-3}$ & 5 & (7581) & 75505 & 2.09 \\ 
 &  & $10^{-1}$ & 7 & (11631) & 115834 & 3.21 \\ 
 & \multirowcell{-3}{Proximal} & $2$ & 34 & (44165) & 439829 & 14.21 \\ 
\hhline{*1{~}|*6{-}|}
\multirowcell{-13}{400} & Gradient & -- & \multicolumn{2}{c}{6860} & 68325 & 5.36 \\ 
\hline\hline

 &  & 0.05 & 1047 & (8044) & 65861 & 2.09 \\ 
 &  & 0.10 & 524 & (4747) & \bf 34172 & \bf 1.13 \\ 
 & \multirowcell{-3}{Broximal-F} & 0.50 & 107 & (8634) & 69178 & 2.33 \\ 
\hhline{*1{~}|*6{-}|}
 &  & $10^{-4}$ & 1065 & (7441) & 55117 & 1.75 \\ 
 &  & $10^{-3}$ & 442 & (8487) & 71393 & 2.57 \\ 
 & \multirowcell{-3}{Broximal-A} & $10^{-2}$ & 4 & (6474) & 64438 & 2.04 \\ 
\hhline{*1{~}|*6{-}|}
 &  & $10^{-3}$ & 4850 & (11042) & 58099 & 2.16 \\ 
 &  & $10^{-2}$ & 1306 & (7862) & 59289 & 2.28 \\ 
 & \multirowcell{-3}{Broximal-P} & $10^{-1}$ & 183 & (7444) & 60155 & 2.04 \\ 
\hhline{*1{~}|*6{-}|}
 &  & $10^{-3}$ & 5 & (7069) & 70407 & 2.36 \\ 
 &  & $10^{-1}$ & 7 & (10523) & 104802 & 3.47 \\ 
 & \multirowcell{-3}{Proximal} & $2$ & 32 & (38933) & 387729 & 14.51 \\ 
\hhline{*1{~}|*6{-}|}
\multirowcell{-13}{500} & Gradient & -- & \multicolumn{2}{c}{6488} & 64621 & 5.76 \\ 
\hline
\end{tabular}
\end{minipage} 
\hspace{20pt}
%==============================================================================
\begin{minipage}[t]{0.45\textwidth}
\begin{tabular}{|c|ccr@{\hspace{0.1em}}ccc|}
\hline
$n$ & Alg. & Param. & \multicolumn{2}{c}{Outer(Inner)} & \#$f$ & time \\ 
\hline

 &  & 0.05 & 1103 & (11281) & 93490 & 4.90 \\ 
 &  & 0.10 & 552 & (5163) & \bf 38093 & \bf 1.72 \\ 
 & \multirowcell{-3}{Broximal-F} & 0.50 & 112 & (8947) & 78871 & 3.50 \\ 
\hhline{*1{~}|*6{-}|}
 &  & $10^{-4}$ & 870 & (8671) & 70363 & 3.01 \\ 
 &  & $10^{-3}$ & 300 & (5590) & 42884 & 1.94 \\ 
 & \multirowcell{-3}{Broximal-A} & $10^{-2}$ & 4 & (6462) & 64318 & 2.56 \\ 
\hhline{*1{~}|*6{-}|}
 &  & $10^{-3}$ & 4898 & (10551) & 53079 & 2.19 \\ 
 &  & $10^{-2}$ & 1247 & (10004) & 78651 & 3.09 \\ 
 & \multirowcell{-3}{Broximal-P} & $10^{-1}$ & 188 & (6113) & 47751 & 2.17 \\ 
\hhline{*1{~}|*6{-}|}
 &  & $10^{-3}$ & 5 & (7043) & 70148 & 4.57 \\ 
 &  & $10^{-1}$ & 7 & (10684) & 106405 & 5.52 \\ 
 & \multirowcell{-3}{Proximal} & $2$ & 32 & (39142) & 389812 & 18.61 \\ 
\hhline{*1{~}|*6{-}|}
\multirowcell{-13}{600} & Gradient & -- & \multicolumn{2}{c}{6453} & 64272 & 5.40 \\ 
\hline\hline

 &  & 0.05 & 1158 & (6181) & 44009 & 1.95 \\ 
 &  & 0.10 & 580 & (7573) & 59648 & 2.55 \\ 
 & \multirowcell{-3}{Broximal-F} & 0.50 & 118 & (9455) & 81296 & 3.44 \\ 
\hhline{*1{~}|*6{-}|}
 &  & $10^{-4}$ & 977 & (6081) & 39453 & 1.78 \\ 
 &  & $10^{-3}$ & 412 & (7469) & 60274 & 2.65 \\ 
 & \multirowcell{-3}{Broximal-A} & $10^{-2}$ & 4 & (7191) & 71577 & 3.01 \\ 
\hhline{*1{~}|*6{-}|}
 &  & $10^{-3}$ & 5105 & (10732) & 50591 & 2.42 \\ 
 &  & $10^{-2}$ & 1349 & (6416) & \bf 38903 & \bf 1.71 \\ 
 & \multirowcell{-3}{Broximal-P} & $10^{-1}$ & 192 & (8751) & 67214 & 3.13 \\ 
\hhline{*1{~}|*6{-}|}
 &  & $10^{-3}$ & 5 & (8146) & 81132 & 3.33 \\ 
 &  & $10^{-1}$ & 7 & (12651) & 125990 & 5.37 \\ 
 & \multirowcell{-3}{Proximal} & $2$ & 36 & (49317) & 491135 & 28.80 \\ 
\hhline{*1{~}|*6{-}|}
\multirowcell{-13}{700} & Gradient & -- & \multicolumn{2}{c}{7193} & 71641 & 9.32 \\ 
\hline\hline

 &  & 0.05 & 1288 & (6700) & \bf 46591 & \bf  2.88 \\ 
 &  & 0.10 & 645 & (6732) & 52296 & 3.09 \\ 
 & \multirowcell{-3}{Broximal-F} & 0.50 & 131 & (11952) & 99591 & 7.56 \\ 
\hhline{*1{~}|*6{-}|}
 &  & $10^{-4}$ & 1033 & (7842) & 53904 & 3.09 \\ 
 &  & $10^{-3}$ & 473 & (9356) & 74873 & 4.08 \\ 
 & \multirowcell{-3}{Broximal-A} & $10^{-2}$ & 4 & (7504) & 74695 & 4.78 \\ 
\hhline{*1{~}|*6{-}|}
 &  & $10^{-3}$ & 5449 & (13510) & 75545 & 6.05 \\ 
 &  & $10^{-2}$ & 1446 & (9375) & 70659 & 4.42 \\ 
 & \multirowcell{-3}{Broximal-P} & $10^{-1}$ & 195 & (12439) & 91864 & 5.78 \\ 
\hhline{*1{~}|*6{-}|}
 &  & $10^{-3}$ & 5 & (8730) & 86946 & 4.75 \\ 
 &  & $10^{-1}$ & 8 & (14672) & 146112 & 8.28 \\ 
 & \multirowcell{-3}{Proximal} & $2$ & 37 & (53698) & 534761 & 32.64 \\ 
\hhline{*1{~}|*6{-}|}
\multirowcell{-13}{800} & Gradient & -- & \multicolumn{2}{c}{7529} & 74986 & 11.23 \\ 
\hline\hline

 &  & 0.05 & 1307 & (10169) & 81702 & 5.51 \\ 
 &  & 0.10 & 655 & (11517) & 96444 & 6.58 \\ 
 & \multirowcell{-3}{Broximal-F} & 0.50 & 133 & (11228) & 93907 & 6.79 \\ 
\hhline{*1{~}|*6{-}|}
 &  & $10^{-4}$ & 966 & (7293) & 52514 & 4.57 \\ 
 &  & $10^{-3}$ & 387 & (4088) & \bf 28231 & \bf 3.43 \\ 
 & \multirowcell{-3}{Broximal-A} & $10^{-2}$ & 4 & (6866) & 68334 & 5.37 \\ 
\hhline{*1{~}|*6{-}|}
 &  & $10^{-3}$ & 5512 & (13959) & 80800 & 7.44 \\ 
 &  & $10^{-2}$ & 1356 & (8504) & 62961 & 5.02 \\ 
 & \multirowcell{-3}{Broximal-P} & $10^{-1}$ & 194 & (8485) & 70221 & 4.65 \\ 
\hhline{*1{~}|*6{-}|}
 &  & $10^{-3}$ & 5 & (7607) & 75760 & 5.03 \\ 
 &  & $10^{-1}$ & 7 & (11759) & 117102 & 9.24 \\ 
 & \multirowcell{-3}{Proximal} & $2$ & 34 & (44386) & 442022 & 37.21 \\ 
\hhline{*1{~}|*6{-}|}
\multirowcell{-13}{900} & Gradient & -- & \multicolumn{2}{c}{6872} & 68441 & 17.17 \\ 
\hline\hline

 &  & 0.05 & 1447 & (7740) & \bf 55806 & \bf 8.51 \\ 
 &  & 0.10 & 725 & (10760) & 87617 & 12.16 \\ 
 & \multirowcell{-3}{Broximal-F} & 0.50 & 147 & (11161) & 92814 & 13.30 \\ 
\hhline{*1{~}|*6{-}|}
 &  & $10^{-4}$ & 1282 & (9534) & 72604 & 10.86 \\ 
 &  & $10^{-3}$ & 562 & (9014) & 70566 & 9.63 \\ 
 & \multirowcell{-3}{Broximal-A} & $10^{-2}$ & 4 & (7430) & 73957 & 9.64 \\ 
\hhline{*1{~}|*6{-}|}
 &  & $10^{-3}$ & 6002 & (12819) & 62685 & 11.15 \\ 
 &  & $10^{-2}$ & 1462 & (8894) & 61259 & 9.44 \\ 
 & \multirowcell{-3}{Broximal-P} & $10^{-1}$ & 196 & (11968) & 92612 & 14.84 \\ 
\hhline{*1{~}|*6{-}|}
 &  & $10^{-3}$ & 5 & (8539) & 85043 & 10.96 \\ 
 &  & $10^{-1}$ & 8 & (14289) & 142297 & 18.84 \\ 
 & \multirowcell{-3}{Proximal} & $2$ & 37 & (52428) & 522109 & 97.86 \\ 
\hhline{*1{~}|*6{-}|}
\multirowcell{-13}{1000} & Gradient & -- & \multicolumn{2}{c}{7419} & 73890 & 40.25 \\ 
\hline

\end{tabular}
\end{minipage}
\caption{Numerical results for selected dimensions and representative parameter values. 
In the column \emph{Param.}, the reported value denotes the fixed radius $t$ for Broximal-F, 
the scaling parameter $\alpha$ for Broximal-A, the scaling parameter $\beta$ for Broximal-P, 
and the regularization parameter $\lambda$ for the Proximal algorithm.}
\label{tab:results}
\end{table}

The main observation is that the Broximal variants are competitive with, and often outperform, both the proximal point method and the direct gradient method in terms of function evaluations and CPU time. Considering, for each dimension, the Broximal variant with the smallest number of function evaluations, the required number of function evaluations was, on average, about \(59\%\) of that required by the Gradient method and about \(53\%\) of that required by the best Proximal variant. No single Broximal radius rule is uniformly superior across all tested dimensions, namely,  depending on the instance, the best performance is attained by Broximal-A, Broximal-P, or Broximal-F. In particular, the results suggest that adaptive radii based either on the gradient norm or on a Polyak-type scaling can provide a useful balance between progress in the outer iteration and the cost of solving the ball subproblem.

The results also show that the number of outer iterations alone is not a reliable measure of efficiency. For instance, in Broximal-A, larger adaptive radii may reduce the number of outer iterations to only a few steps, but at the price of increasing the cost of the inner solves. This is observed for \(\alpha=10^{-2}\), for which the number of outer iterations drops to four in all tested dimensions, while the total number of function evaluations becomes comparable to that of the direct Gradient method, indicating that most of the computational effort is transferred to the subproblem solver. A related effect is also observed in Broximal-P, where a reduction in outer iterations does not automatically translate into a lower overall cost. Thus, the total number of function evaluations and the CPU time provide a more meaningful comparison. Figure~\ref{fig:funeval} illustrates this point by reporting the number of function evaluations as a function of the dimension for one representative parameter choice of each method. The curves corresponding to the adaptive Broximal variants, namely Broximal-A with \(\alpha=10^{-3}\) and Broximal-P with \(\beta=10^{-2}\), are frequently below the Gradient and Proximal curves. This indicates that adaptive radius strategies can lead to a lower computational cost in terms of function evaluations.

\begin{figure}[h] \centering
\includegraphics[scale=0.45]{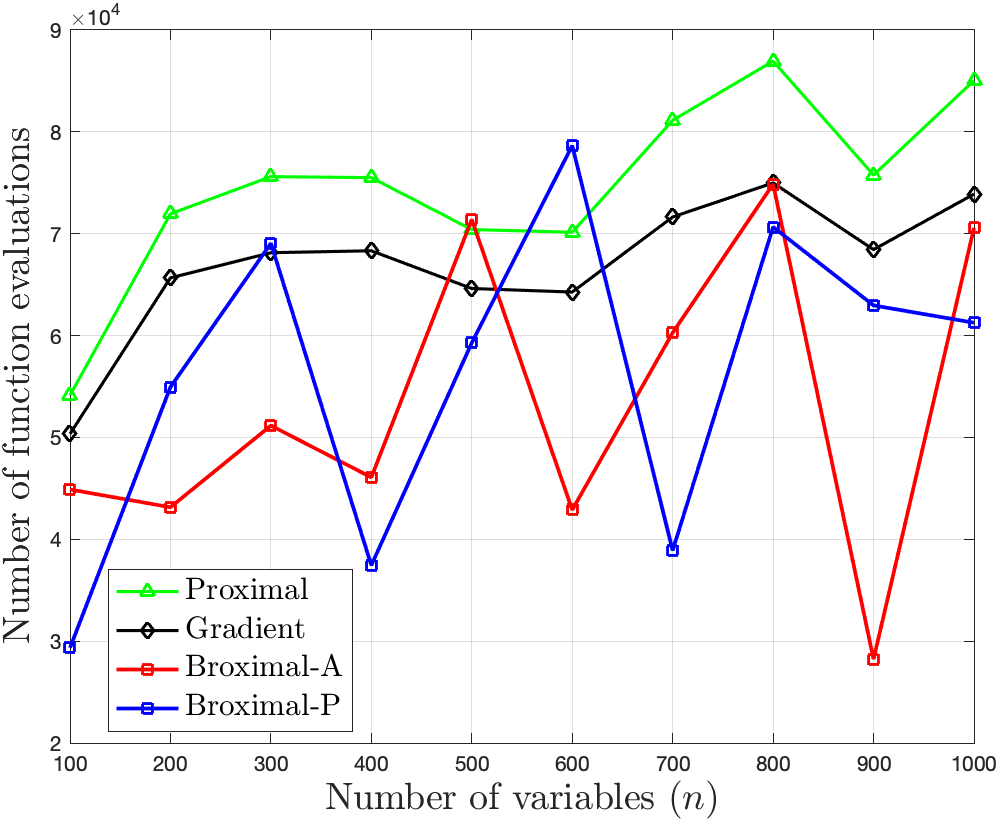}
\caption{Number of function evaluations for representative parameter choices: Broximal-A ($\alpha=10^{-3}$), Broximal-P ($\beta=10^{-2}$), the Proximal method ($\lambda=10^{-3}$), and the Gradient method. Lower curves indicate lower computational cost.}
\label{fig:funeval}
\end{figure}

The fixed-radius variant Broximal-F also performs well for suitable choices of the radius, and in some dimensions it gives the best overall performance. This suggests that both fixed and variable radius strategies can be effective, although their performance depends on the specific rule and parameter choice. In contrast, the Proximal point method is more sensitive to the regularization parameter,  small values of \(\lambda\) lead to the best results among the tested choices, whereas larger values substantially increase the cost of the inner solves.

We now complement the results with a qualitative illustration. Figure~\ref{fig:traj} compares the sequence of points generated by the Gradient method with the sequence of points produced by the Broximal-A scheme on a two-dimensional quadratic problem. At each outer iteration, the corresponding ball-constrained subproblem is approximately solved by projected gradient, initialized at the last iterate obtained for the previous subproblem. Consequently, the plotted Broximal-A path includes both the outer iterates and the intermediate projected-gradient iterates arising in the inner solves. The figure suggests that the ball-constrained mechanism may attenuate the zig-zag pattern typical of gradient descent in ill-conditioned quadratic valleys.

\begin{figure}[h!] \centering
\begin{tabular}{cc}
(a) Gradient & (b) Broximal-A\\
\includegraphics[scale=0.45]{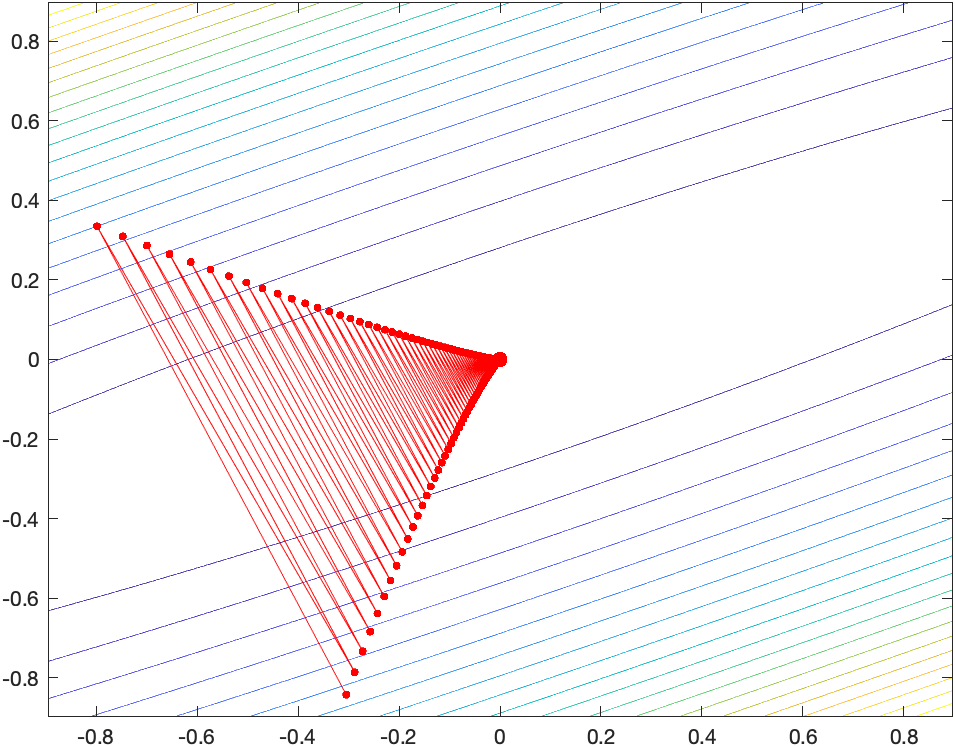} & \includegraphics[scale=0.45]{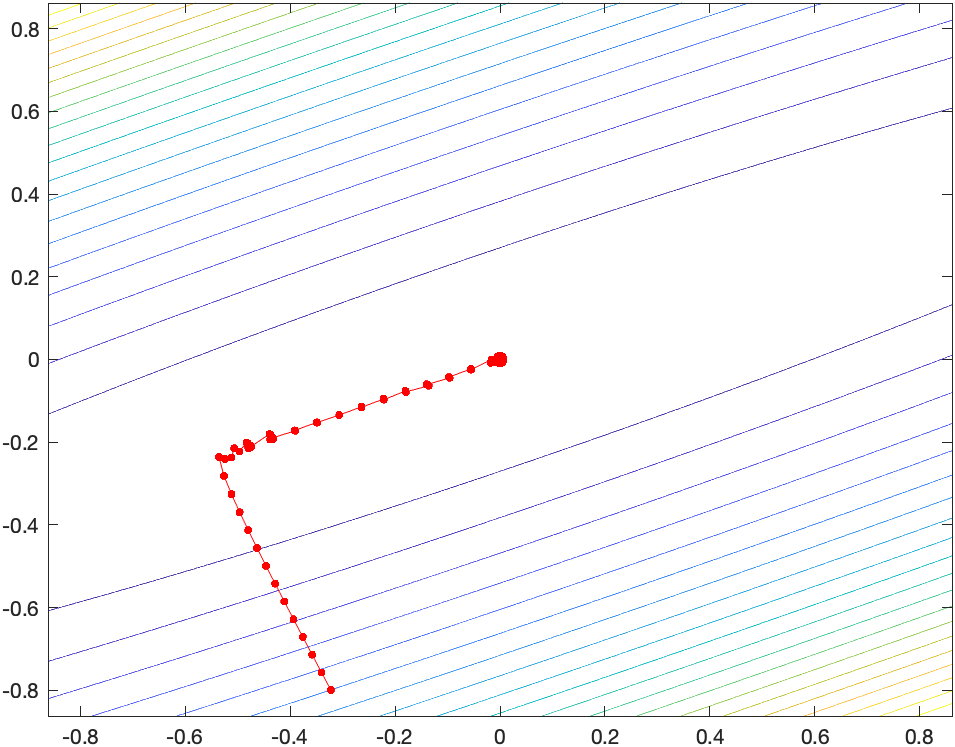} \\
\end{tabular}
\caption{Illustrative pointwise paths on a two-dimensional strongly convex quadratic problem. (a) Iterates generated by the Gradient method with Armijo line search. (b) Points generated by the Broximal-A implementation, including both the outer iterates and the intermediate iterates produced by the projected-gradient solver used to approximately solve each ball-constrained subproblem, initialized from the last iterate of the previous subproblem.}
\label{fig:traj}
\end{figure}

Overall, the preliminary experiments in this section suggest that inexact ball-proximal schemes behave in a stable and competitive way on this simple Euclidean quadratic benchmark. Of course, the purpose here is not to advocate these methods as practical solvers for quadratic problems, for which more specialized approaches are available, but rather to use a classical and controlled setting to isolate their algorithmic behavior. In this perspective, the experiments highlight that the choice of the radius plays a central role in performance. A fixed radius may be effective when properly tuned, but its choice can be delicate because the appropriate scale depends on the problem. The adaptive strategies, by adjusting the radius according to first-order information, provide a simple way to incorporate local information into the step-size control, which makes them promising candidates for future studies in more general settings, where a suitable fixed radius may be harder to prescribe.

%%%%%%%%%%%%%%%%%%%%%%%%%%%%%%%
\section{Final remarks} \label{Sec:SolSupProRB-PPM}

The RB--PPM may be viewed as a ball-constrained counterpart of the classical Riemannian proximal point method. In the latter, one solves the regularized subproblem \eqref{eq:ppm}, whereas RB--PPM replaces the quadratic penalization by the ball-constrained step \eqref{eq:bpmsub}. Thus, both methods control the iterates by modifying the original problem, but one does so through penalization and the other through localization.

What is distinctive in the ball-proximal framework is that the parameter \(t_k\) has a direct geometric meaning, namely,  it is precisely the radius of the region in which the next iterate is sought. In particular, when \(\mathbb{B}(p_k,t_k)\cap\Omega^*=\varnothing\), Theorem~\ref{thm:first-brox-Riem} shows that the broximal point is uniquely defined on the sphere \(\mathbb{S}(p_k,t_k)\) and satisfies \(d(p_{k+1},p_k)=t_k\). Combined with the distance decrease \eqref{eq:distance-drop}, this yields the finite-termination result of Theorem~\ref{thm:icfv}, the linear and \(O(1/K)\) function-value bounds in Theorems~\ref{th:linear} and~\ref{thm:D3}, and the asymptotic dichotomy of Theorem~\ref{thm:dichotomy}. In this sense, RB--PPM provides a complementary perspective to classical proximal regularization, emphasizing geometric features that are specific to ball-constrained updates. At the same time, the method studied here remains idealized, since each iteration requires the exact minimization of \(f\) over a geodesic ball. Its main role is therefore conceptual. In fact,  it serves as a theoretical benchmark for understanding what can be expected from algorithms that only approximate ball-constrained subproblems.  A natural next step is to develop implementable inexact variants, in which each broximal subproblem is solved only approximately by a finite number of inner iterations of a suitable Riemannian method, together with error criteria compatible with the outer convergence analysis. It would also be important to investigate new adaptive strategies for choosing the radii, balancing progress in the outer iteration with the computational cost of the inner solves, as well as to analyze their impact on convergence and complexity guarantees.

Among possible radius choices, a natural candidate is the Polyak-type rule
\begin{equation} \label{eq:psz}
t_k=\min\Bigl\{t_{\max},\max\bigl\{t_{\min},\beta\,\tfrac{f(p_k)-\ell_k}{\|v_k\|}\bigr\}\Bigr\},
\qquad v_k\in\partial f(p_k),
\end{equation}
where \(\beta>0\) and \(\ell_k\le f^*\) is a lower estimate of the optimal value. This type of choice is consistent with step-length rules commonly used in practical implementations of subgradient methods with Polyak stepsizes and related level-type strategies, where the displacement is scaled by the ratio between an objective gap estimate and the norm of a first-order residual.  This is the natural ball-radius analogue of the classical Polyak stepsize, since in RB--PPM the radius plays the role of a step-length parameter and exact boundary steps satisfy \(d(p_{k+1},p_k)=t_k\). Moreover, under convexity, item~(i) of Proposition~\ref{pr:f-convex-subdiff} yields
\[
0<\frac{f(p_k)-f^*}{\|v_k\|}\le \operatorname{dist}(p_k,\Omega^*),
\]
so the rule automatically produces smaller radii near the solution set and allows larger displacements only when the objective gap is significant relative to the norm of a first-order residual. This suggests that adaptive radius rules of Polyak type may be a promising ingredient in future practical implementations of RB--PPM.

Finally, it would be interesting to extend the present analysis beyond the convex setting and to investigate whether analogous geometric, asymptotic, and complexity properties remain valid for suitable nonconvex formulations on a Hadamard manifolds.

%%%%%%%%%%%%%%%%%%%%%%%%%%%%%%%
\bibliographystyle{abbrv}
\bibliography{BroximalRiemannian}

\end{document}